\documentclass[reqno,12pt]{amsart}

\usepackage{amsmath,amssymb,latexsym}

\newcommand{\nm}[2]{\Vert #1\Vert _{#2}}

\newcommand{\ep}{\varepsilon}
\newcommand{\fy}{\varphi}

\newcommand{\cdo}{\, \cdot \, }

\newcommand{\supp}{\operatorname{supp}}

\newcommand{\eabs}[1]{\langle #1\rangle}

\newcommand{\vrum}{\vspace{0.1cm}}

\def\cN{\mathcal{N}}

\def\R{\mathbb{R}}

\def\N{\mathbb{N}}
\def\sch{\mathcal{S}}

\def\la{\langle}
\def\ra{\rangle}

\def\Ren{\mathbb{R}^d}
\def\Renn{\mathbb{R}^{2d}}

\def\cD{\mathcal{D}}

\def\cS{\mathcal{S}}

\def\o{\omega}
\def\cM{\mathcal{M}}

\def\inv{^{-1}}

\def\Mmpq{M_m^{p,q}}
\def\phas{(x,\omega )}

\def\a{\alpha}                    % Greek Letters

\newcommand{\stft}{short-time Fourier transform}
\def\Fur{\mathcal{F}}
\def\be{\begin{equation}}
\def\ee{\end{equation}}
\def\bena{\begin{eqnarray*}}
\def\ena{\end{eqnarray*}}

 \def\mN{\mathbb{N}}  \def\mC{\mathbb{C}}

\def\phas{(x,\omega )}
\newcommand{\modsp}{modulation space}

\def\cD{\mathcal{D}}
\def\cS{\mathcal{S}}

\DeclareMathOperator{\WF}{WF}
\def\Mopq{M_{(\omega)} ^{p,q}}
\newcommand{\MRs}{\mathcal M_{\{s\}}}

\newcommand{\rr}[1]{\mathbf R^{#1}}

\newcommand{\sets}[2]{\{ {\,}#1{\,};{\,}#2{\,}\} }

\theoremstyle{plain}
 \newtheorem{theorem}{Theorem}[section]
 \newtheorem{proposition}{Proposition}[section]
 \newtheorem{lemma}{Lemma}[section]
 
\theoremstyle{definition}
 
 \newtheorem{definition}{Definition}[section]
\theoremstyle{remark}
 \newtheorem{remark}{Remark}[section]
 \numberwithin{equation}{section}

\begin{document}

\title[]
{Wave-front sets in non-quasianalytic setting for Fourier Lebesgue and modulation spaces}

\author[N. Teofanov]{Nenad Teofanov}
\address{N. Teofanov\\ University of Novi Sad, Faculty of Sciences\\ Department
 of Mathematics and Informatics\\ Trg Dositeja Obradovi\' ca 4\\ 21000 Novi Sad \\ Serbia }
\email{nenad.teofanov@dmi.uns.ac.rs}

\subjclass[2010]{Primary 35A18,35S30,42B05,35H10}

\keywords{Wave-front sets, weighted Fourier-Lebesgue spaces, Gelfand-Shilov spaces, ultradistributions}

\begin{abstract}
We define and study wave-front sets for weighted Fourier-Lebesgue spaces when the weights are moderate with respect to the associated functions
for general sequences $\{ M_p\} $ which satisfy Komatsu's conditions $(M.1) - (M.3)'$.
In particular, when  $\{ M_p\} $ is the Gevrey sequence ($M_p = p!^s$, $s>1$) we recover some previously observed results.
Furthermore, we consider  wave-front sets for modulation spaces in the same setting, and prove the invariance property related to  the Fourier-Lebesgue type  wave-front sets.
\end{abstract}

\maketitle

\section{Introduction}

Wave-front sets in the context of  Fourier-Lebesgue spaces, together with the study of corresponding pseudodifferential operatros,  were first considered in \cite{PTT3}, see also \cite{PTT, PTT2,PTT4}.
They are recently used in \cite{CGN} for a mathematical explanation of phenomena related to the interferences in the Born-Jordan distribution.
The conic neighborhoods in the definition of such wave-front sets are replaced in \cite{GM}
by a filter of neighborhoods for the study of propagation of singularities of Fourier-Lebesgue type for partial (pseudo)differential equations, whose symbol satisfies generalized elliptic properties.
An important extension  of investigations from \cite{PTT, PTT2} to general  weighted Fourier Banach spaces is given in \cite{CJT1, CJT2}.

The above mentioned results are performed in the framework of weights of polynomial growth and, consequently, within the realm of tempered distributions. Spaces of ultradistributions in the context of weighted Fourier-Lebesgue type spaces 
were first observed in \cite{JPTT2}, see also \cite{JPTT3}. The sequences of the form $ M_p = p!^s,$ $s>1$, are used there to define the corresponding test function spaces. This in turn
leads to the analysis of weighted  Fourier-Lebesgue spaces such that the growth of the weight function at infinity is bounded by $e^{k|\cdot|^{1/s}},$ for some $k>0.$

In this paper we extend the results from \cite{JPTT2} to a more general context when the spaces of test functions are given by the means of 
$\{ M_p\} $ sequences which satisfy Komatsu's conditions $(M.1) - (M.3)'$, see Section \ref{sec1}. Note that this allows "fine tuning"
between the two Gevrey type sequences, see Remark \ref{fino tuniranje}.

The paper is organized as follows. We end the introduction with the basic notation, and a brief account on weight functions. Section\ref{sec1}
contains a discussion on sequences and corresponding associate functions, which are the basic notions in our analysis. We proceed with an exposition of Gelfand-Shilov spaces and other test function spaces, and their dual spaces of ultradistributions. Section \ref{sec2} contains the definition of  wave-front sets for weighted Fourier-Lebesgue spaces when the weights are submultiplicative with respect to the associated function of a given non-quasianalytic sequence $\{ M_p\}.$ We study its basic properties, convolution relations, and discuss its relation to some other types of  wave-front sets. In Section \ref{sec3} we first study the short-time Fourier transform in the context of test function spaces and their duals from Section \ref{sec1}, and then define modulation spaces and recall their basic properties. Finally, in Section \ref{sec4} we introduce wave-front sets for modulation spaces and show that they coincide with appropriate wave-front sets from Section \ref{sec2}. Since we consider general  non-quasianalytic sequences $\{ M_p\}$, we recover the main results from \cite{JPTT2, JPTT3} where the particular case $ M_p = p!^s $, $s>1$, is observed.

%%%%%%%%%%%%%%%%%%%%%%%%%%%%%%%%%%%%%%%%%%%%%
\subsection{Basic notation} \label{notions}
%%%%%%%%%%%%%%%%%%%%%%%%%%%%%%%%%%%%%%%%%%%%%
We put $\mathbb{N} =\{0,1,2,\dots \}$, $\eabs x =(1+|x|^2)^{1/2}$,
 $x\in \mathbb{R}^d$,
$xy=x\cdot y$ denotes  the scalar product on $\mathbb{R}^d$ and
$$%\begin{equation}
 \la \phas \ra^s= \la z\ra^s=(1+x^2+\o^2)^{s/2},\quad
   z=(x,\o)\in\Renn, \,\quad s\in\mathbb{R}.
   % \label{eqc1}
 %\end{equation}
$$  
The partial derivative
of a vector $x=(x_1,\dots,x_d)\in \mathbb{R}^d$ with respect to
$x_j$ is denoted by  $\partial _j = \frac{\partial}{\partial x_j}$.
Given a multi-index
$p=(p_1,\dots,p_d)\geq 0$, i.e., $p \in\mathbb{N}^d_0$ and $p_j \geq 0$, we write
$\partial^p=\partial^{p_1}_1\cdots\partial^{p_d}_d$ and  $x^p =(x_1,\dots,x_d)^{(p_1,\dots,p_d)}=\prod_{i=1}^d x_i^{p_i}$.
Similarly, $h\cdot |x|^{1/\a}=\sum_{i=1}^d h_i |x_i|^{1/\a_i}$. Moreover, for $p \in\mathbb{N}^d_0$ and $\a \in\R^d_+$, we set $(p!)^\a=(p_1!)^{\a_1}\dots (p_d!)^{\a_d}$.
In the sequel, a real number $r\in\R_+$ may play the role of the vector
with constant components $r_j=r$, so for $\a\in\R^d_+$, by writing $\a>r$ we mean $\a_j>r$ for all $j=1,\dots,d$.
By $X$ we denote an open set in $ \mathbb{R}^d,$ and $ K \Subset X $ means that $K$ is  compact subset in
$X.$

The Fourier transform is normalized to be $${\hat   {f}}(\o)=\Fur f(\o)=\int f(t)e^{-2\pi i t\o}dt.$$
We use the brackets $\la f,g\ra$ to denote the extension of the inner product $\la f,g\ra=\int f(t){\overline
{g(t)}}dt$ on $L^2(\Ren)$ to the dual pairing between a test function space $ \mathcal A $ and its dual $ {\mathcal A}' $:
$ \langle \cdot, \cdot \rangle  = $ $ _{{\mathcal A}'}\langle \cdot, \overline{\cdot} \rangle _{\mathcal A}.$
We use the standard notation for usual spaces of functions and distributions, e.g. $L^p (\mathbb{R}^d)$, $L^p _{loc }(\Omega)$,
$ 1 \leq p \leq \infty $, denote Lebesgue spaces and their local versions respectively,
$\mathcal{S} (\mathbb{R}^d)$ denotes the Schwartz space of rapidly decreasing test functions, etc.

\par

Translation and modulation operators, $T$ and $M$ respectively, when acting on $ f \in L^2 (\mathbb{R}^d)$
are defined by
\begin{equation} \label{trans-mod}
T_x f(\cdot) = f(\cdot - x) \;\;\; \mbox{ and } \;\;\;
 M_x f(\cdot) = e^{2\pi i x \cdot} f(\cdot), \;\;\; x \in \mathbb{R}^d.
\end{equation}
Then for $f,g \in L^2 (\mathbb{R}^d)$ the following relations hold:
$$
M_y T_x  = e^{2\pi i x \cdot y } T_x M_y, \;\;
 (T_x f)\hat{} = M_{-x} \hat f, \;\;
  (M_x f)\hat{} = T_{x} \hat f, \;\;\;
  x,y \in \mathbb{R}^d.
$$
These operators are extended to other spaces of functions and distributions in a natural way.

Throughout the paper, $A\lesssim B$
denotes $A\leq c B$ for a suitable constant $c>0$, whereas $A
\asymp B$ means that $c\inv A \leq B \leq c A$ for some $c\geq 1$. The
symbol $B_1 \hookrightarrow B_2$ denotes the continuous and dense embedding of
the topological vector space $B_1$ into $B_2$.

%%%%%%%%%%%%%%%%%%%%%%%%%%%%%%%%%%%%%%%%%%
\subsection{Weights}
%%%%%%%%%%%%%%%%%%%%%%%%%%%%%%%%%%%%%%%%%%

\par

In general, a weight function is a non-negative function in $L^\infty _{loc}$.

\par

\begin{definition}
Let $\omega, v$ be non-negative functions. Then %%
\begin{enumerate}
\item $v$ is called submultiplicative if
\[
v(x+y)\leq v(x)v(y), \qquad \forall \ x,y\in\mathbb{R}^d;
\]
\item $\omega$ is called $v$-moderate if %there exists $C>0$ such that
\[
\omega(x+y) \lesssim v(x)\omega (y), \qquad \forall \ x,y\in\mathbb{R}^d.
\]
\end{enumerate}
For a given submultiplicative weight $v$  the set of all
$v$-moderate weights will be denoted by $\mathcal{M}_{v}$.
\end{definition}

\par

If $v$ is even and $\omega \in \mathcal{M}_{v} $, then $1/v \lesssim \omega \lesssim v$,  $\omega \neq 0$
everywhere and $1/\omega \in \mathcal{M}_{v}$.

\par

In the sequel we assume that $v$ is an even submultiplicative function.
Submultiplicativity implies that $v$ is dominated by an exponential function, i.e.
\begin{equation*} %\label{weight}
v \leq C e^{k |\cdo|} \quad \mbox{for some}\quad C, k>0.
\end{equation*}

For example, every weight of the form
$$ %\begin{equation} \label{BDweight}
v(z) =   e^{s\|z\|^b} (1+\|z\|)^a \log ^r(e+\|z\|)
$$ %\end{equation}
for parameters $a,r,s\geq 0$, $0\leq b \leq 1$ satisfies the above conditions.

\par

%A submultiplicative weight $v$ satisfies the GRS condition (the
%Gelfand-Raikov-Shilov condition) if $ \displaystyle \lim_{n\to
%\infty} v(n x)^{1/n} =1$, for every $x\in \mathbb{R}^d$.
%
%%or, equivalently, if $ \displaystyle\lim_{n\to\infty} \frac{\log v(n\cdo)}{n}=0$.
%
%\par

Let $s>1$.
By $\MRs (\mathbb{R}^d)$ we denote the set of  all weights which are moderate
with respect to a weight $v$ which satisfies $ v \leq Ce^{k|\cdo|^{1/s}}$
for some positive constants $C$ and $k$.
The weight $v$ satisfy the Beurling-Domar non-quasi-analyticity condition which takes the form
\[
\sum\limits_{n=0 } ^{\infty}\frac{\log v(nx)}{n^2} < \infty, \quad  x \in \mathbb{R}^d.
\]
We refer to \cite{Gro2} for a detailed account on weights in time-frequency analysis.

\par

\section{Spaces of test functions and their duals} \label{sec1}

Let $ (M_p)_{p \in \mN_0} $ be a sequence of positive numbers monotonically
increasing to infinity which satisfies:

\noindent $(M.1)\;\;$ $ M_p ^2 \leq M_{p-1} M_{p+1}, \;\;\; p \in
\mN; $

\noindent $(M.2) \;\;$ There exist positive constants $ A,H $ such
that
$$
 M_{p} \leq A H^p \mbox{ min }_{0 \leq q \leq p} M_{p-q} M_{q},
\;\; p,q \in \mN_0,
$$
or, equivalently, there exist positive constants $ A,H $ such that
$$
 M_{p+q} \leq A H^{p+q} M_{p} M_{q},
\;\; p,q \in \mN_0;
$$
\noindent $(M.3)' \;\;$ $ \sum_{p=1} ^{\infty} M_{p-1}/M_p  < \infty.$

We assume that $ M_0 = 1, $ and that $ M_p ^{1/p } $ is bounded below by
a positive constant.

The condition $(M.3)' $ provides the existence of nontrivial compactly supported smooth functions (and therefore partitions of unity)
in the corresponding spaces of test functions. It is therefore known as the  non-quasianalyticity condition.

The  Gevrey sequences $ M_p = p!^{s},$ $ p \in \mN,$ $ s > 1 $, are  basic examples of sequences which satisfy $(M.1) - (M.3)' $.

Let $ (M_p)_{p \in \mN_0} $ and $ (N_q)_{q \in \mN_0} $ be sequences
which satisfy $ (M.1). $ We write $ M_p \subset N_q $ ($ (M_p) \prec
(N_q), $ respectively) if there are constants $ H,C > 0 $ (for any
$H>0$ there is a constant $ C>0,$ respectively) such that $ M_p \leq
C H^p N_p, $ $ p \in \mN_0.$ Also, $ (M_p)_{p \in \mN_0} $ and $
(N_q)_{q \in \mN_0} $ are said to be equivalent if $ M_p \subset N_q
$ and $ N_q \subset M_p $ hold.

\begin{remark} \label{fino tuniranje}
The conditions (M.1) and (M.2) can be described  as follows. Let
$(s_p)_{p \in \mN_0}$ be a sequence of positive numbers monotonically increasing to infinity
($ s_p \nearrow \infty $) so that for every $p, q
\in \mathbb{N}_0$ there exist $A, H>0$ such that
\begin{equation} \label{(M.2)}
\prod_{j=1} ^{q} s_{p+j} =  s_{p+1} \cdots s_{p+q}\leq AH^{p} s_{1} \cdots s_{q} =  AH^{p} \prod_{j=1} ^{q} s_{j}.
\end{equation}
Then  the sequence $ (S_p)_{p \in \mN_0} $ given by $S_p =  \prod_{j=1} ^{p} s_j$, $S_0=1,$
satisfies  $(M.1)$ and $(M.2) $.

Conversely, if $ (S_p)_{p \in \mN_0} $ given by $S_p =  \prod_{j=1} ^{p} s_j$, $s_j > 0, $ $ j \in \mathbb{N},$ $S_0=1,$
satisfies $(M.1)$ then the sequence $ (s_p)_{p \in \mN_0} $ increases to
infinity. If, in addition, it satisfies $(M.2)$ then (\ref{(M.2)}) holds.

\par

Furthermore, if $ ( M_p )_{p \in \mN_0} $  and $ ( N_q )_{q \in \mN_0} $
are given by
\begin{equation} \label{big-seq-cond}
M_p := p!^{\frac{1}{2}} \prod_{k =0} ^p l_k =  p!^{\frac{1}{2}}
L_{p}, \;\;\; p \in \mN_0,\;\;\; N_q := q!^{\frac{1}{2}} \prod_{k
=0} ^q r_k =  q!^{\frac{1}{2}} R_q, \;\;\; q \in \mN_0
\end{equation}
where  $(r_p)_{p \in \mN_0}  $ and $(l_p)_{p \in \mN_0} $
are sequences of positive numbers monotonically increasing to infinity
such that \eqref{(M.2)} holds with the letter $s$ replaced by $r$ and $l$ respectively,
and which satisfy: For
every $\alpha\in(0,1]$ and every $k>1$ so that $kp\in\mathbb{N},
p\in \mathbb{N},$
\begin{equation} \label{seq-cond}
\max\{(\frac{r_{kp}}{r_p})^2,(\frac{l_{kp}}{l_p})^2 \} \leq
k^\alpha, \;\;\; p\in\mathbb{N}.
\end{equation}
Then $ p! \prec M_p N_p$ and the sequences
 $ (R_p)_{p \in \mN_0} $ and $ (L_p)_{p \in \mN_0} $ ($R_p=r_1
\cdots r_p$,  $L_p=l_1  \cdots l_p, p\in \mathbb{N}$ $R_0=1,$ and $L_0=1$)
satisfy  $(M.1)$ and $(M.2) $. Moreover,
$$\max\{R_p,L_p\}\leq p!^{\alpha/2}, p\in\mathbb{N},$$
for every $\alpha\in (0,1]$.
(For $ p,q,k \in \N_{0} ^d $ we have
$ L_{|p|} =  \prod_{|k| \leq |p|} l_{|k|},$ and $ R_{|q|} =
\prod_{|k| \leq |q|} r_{|q|}.$)
Such sequences are used in the study of localization operators in the context of quasianalytic spaces in \cite{CPRT2}.

\end{remark}

The {\em associated function} for a given sequence  $ (M_p) $ is defined by
\begin{equation} \label{asocirana}
M(\rho) = \sup_{p\in \mN} \ln_+ \frac{\rho^p M_0}{M_p}, \;\;\; 0 <
\rho < \infty,
\end{equation}
where $ \ln_+ t := \max \{ \ln t ,0 \}, $ $ t>0.$ It is a non-negative monotonically increasing function which vanishes for sufficiently small $\rho,$ and tends to infinity faster than $ \ln \rho ^p,$ as $\rho \rightarrow \infty.$ Moreover, if $(M_p) $ satisfies  $(M.1)$ and $ (M.3)' $, then $k^p p! / M_p \rightarrow 0$ as $p \rightarrow \infty.$

For example, the associated function for the Gevrey sequence  $ M_p = p!^{s}, $ $ p\in \mN_{0}$, $ s > 1 $, behaves
at infinity  as $ |\cdot|^{1/s},$ cf. \cite{Pe}. In fact, the interplay between the defining sequence and its associated function
plays an important role in the theory of ultradistributions.

The following result will be intensively used in this paper. We refer to \cite{CKP} for its proof.

\begin{lemma} \label{osobineasocirane}
Let there be given sequence $ (M_p) $ which satisfies  $(M.1)$. Then
\begin{equation} \label{konveksiti}
M(\sum_{k=1} ^n \rho_k) \leq \sum_{k=1} ^n M(\rho_k), \;\;\; \rho_k>0, \; k = 1,\dots,n.
\end{equation}
If, in addition, $ (M_p) $ satisfies  $(M.2)$, then
\begin{equation} \label{svojstvoasocirane1}
2 M( \rho) \leq  M(H \rho) + \ln_+ (A) , \;\;\; \rho>0,
\end{equation}
where $A$ and $H$ are the constants in  $(M.2)$> Furthermore, if $L\geq 1,$ then there is a constant $C>0$ such that
\begin{equation} \label{svojstvoasocirane2}
 M( L \rho) \leq \frac{3}{2} L M( \rho) + C, \;\;\; \rho>0,
\end{equation}
and there is a constant $B>0$ and a constant $K_L >0$ which depends on $L$, such that
\begin{equation} \label{svojstvoasocirane3}
LM( \rho) \leq  M( B^{L-1} \rho) + K_L, \;\;\; \rho>0.
\end{equation}
\end{lemma}

\begin{remark} \label{observation} By Lemma \ref{osobineasocirane}, it follows that  estimates of the form $ |f(\cdot)| \lesssim e^{M(h|\cdot|)}$ for some/every $h>0$ and $ |f(\cdot)| \lesssim e^{k M(|\cdot|)}$ for some/every $k>0$ are equivalent. This observation will be often used in proofs.
\end{remark}

\subsection{Gelfand-Shilov spaces}
We give here only the basic properties and refer to \cite{GS, NR} for a more detailed discussion and applications in partial differential equations.

\begin{definition} \label{GSoftypeS}
Let there be given sequences of positive numbers $ (M_p)_{p \in
\mN_0} $ and $ (N_q)_{q \in \mN_0} $ which satisfy $ (M.1) $ and $
(M.2).$ Let $ \cS^{N_{q} , B} _{M_{p}, A}
(\R^d) $ be defined by
$$
 \cS^{N_{q} , B} _{M_{p}, A} (\R^d) =
\{ f \in C^\infty(\R^d) \; |\;
 \| x^{\alpha} \partial^{\beta}  f  \|_{L^\infty} \leq
 C A^{\alpha} M_{|\alpha|} B^{\beta} N_{|\beta|}, \;\;
\forall \alpha,\beta \in \N_{0} ^d \},
$$
for some positive constant $C,$ and $ A = (A_1,\dots,A_d),$ $ B =
(B_1,\dots,B_d),$ $A, B>0.$

Gelfand-Shilov spaces $ \Sigma^{N_q} _{M_p} (\R^d) $ and $  \cS^{N_q} _{M_p} (\R^d) $
are projective and inductive limits
of  (Fr\'echet) spaces $ \cS^{N_{q} , B} _{M_{p}, A}
(\R^d) $ with respect to $A$ and $B$:
$$
\Sigma^{N_q} _{M_p} (\R^d) : = {\rm proj} \lim_{A>0, B>0} \cS^{N_q , B}
_{M_p, A} (\R^d) ; \;\;\; \cS^{N_q} _{M_p} (\R^d)  : = {\rm ind} \lim_{A>0, B>0}
\cS^{N_q , B} _{M_p, A}  (\R^d).
$$

The corresponding dual spaces of $\Sigma^{N_q} _{M_p} (\R^d) $ and $\cS^{N_q} _{M_p} (\R^d) $
are the spaces of ultradistributions of Beurling and Roumieu type respectively:
$$
(\Sigma^{N_q} _{M_p} )' (\R^d) : = {\rm ind} \lim_{A>0, B>0}
(\cS^{N_q , B} _{M_p, A})' (\R^d) ;
$$
$$
(\cS^{N_q} _{M_p})' (\R^d)  : = {\rm proj} \lim_{A>0, B>0}
(\cS^{N_q , B} _{M_p, A})'  (\R^d).
$$
\end{definition}

\par

Gelfand-Shilov spaces are closed under translation, dilation,
multiplication with $x\in \R^d,$ and differentiation. Moreover, they are closed under the action of certain differential
operators of infinite order (ultradifferentiable operators in the terminology of Komatsu).

Whenever nontrivial, Gelfand-Shilov spaces contain "enough functions" in the following sense.
A test function space $ \Phi $ is "rich enough" if
$$ \int f(x) \varphi (x) dx = 0, \;\;\; \forall   \varphi \in \Phi
\Rightarrow f(x) \equiv 0 \;\; (a.e.).
$$

\par

The following theorem enlightens the fundamental  properties of Gelfand-Shilov spaces implicitly contained in their definition.
Among other things, it states that the decay and regularity estimates of $f \in
{\mathcal S}^{N_q} _{M_p} (\R^d)$ can be studied separately.

\begin{theorem} \label{GS-characterization}
Let there be given sequences of positive numbers $ (M_p)_{p \in
\mN_0} $ and $ (N_q)_{q \in \mN_0} $ which satisfy $ (M.1) $, $
(M.2)$ and $ p!  \subset M_p N_p $ ($ p! \prec M_p N_p $,
respectively). Moreover, let $ M(\cdot) $ and $ N(\cdot) $ denote the associated functions for  
$ (M_p)_{p \in \mN_0} $ and $ (N_q)_{q \in \mN_0} $ respectively.
Then the following conditions are equivalent:
\begin{enumerate}
\item $f \in {\mathcal S}^{N_q} _{M_p} (\R^d)$
($ f \in \Sigma^{N_q} _{M_p}(\R^d), $ respectively).
\item  There exist constants $A,B\in \R^d,$ $ A,B >0$
(for every   $A,B\in \R^d,$ $ A,B >0$ respectively), and there exist
$ C>0$ such that
$$
\| e^{M(|Ax|)} \partial^{q}  f (x) \|_{L^\infty} \leq C B^{q} N_{|q|}, \quad\forall
p,q\in \N^d_0.
$$
\item There exist constants $A,B\in \R^d,$ $ A,B >0$
(for every  $A,B\in \R^d,$ $ A,B >0,$ respectively), and there exist
$ C>0$ such that
$$ \| x^{p}  f (x) \|_{L^\infty} \leq C A^{p} M_{|p|} \quad\mbox{and}\quad
\| \partial^{q}  f (x) \|_{L^\infty} \leq C B^{q} N_{|q|}, \quad\forall
p,q\in \N^d_0. $$
\item There exist constants $A,B\in \R^d,$ $ A,B >0$
(for every  $A,B\in \R^d,$ $ A,B >0,$ respectively), and there exist
$ C>0$ such that
$$
\| x^{p}  f (x) \|_{L^\infty} \leq C A^{p} M_{|p|} \quad\mbox{and}\quad
\| \omega^{q}  \hat{f} (\omega) \|_{L^\infty} \leq C B^{q} N_{|q|}, \quad\forall
p,q\in \N^d_0.
$$
\item There exist constants $A,B\in \R^d,$ $ A,B >0$
(for every  $A,B\in \R^d,$ $ A,B >0,$ respectively), such that
$$ \|  f(x) e^{M(|Ax|)}\;  \|_{L^\infty} < \infty \quad\mbox{and}\quad
\| \hat f (\o) \; e^{N(|B\o|)} \|_{L^\infty} < \infty. $$ 
\end{enumerate}
\end{theorem}

Theorem \ref{GS-characterization} is proved in \cite{CCK} and reinvented many times afterwards,
see e.g. \cite{CPRT1, GZ, KPP, NR,  PT1,  Teofanov2018}.

By the above characterization $ {\mathcal F}  {\mathcal S}^{N_q} _{M_p}(\mathbb{R}^d)  =  {\mathcal S}^{M_p} _{N_q} (\mathbb{R}^d).  $
When $ M_p = N_q $ we put  $ {\mathcal S}^{M_p} _{M_p} (\mathbb{R}^d) = {\mathcal S}^{\{M_p\}}  (\mathbb{R}^d) $, and
$ {\Sigma}^{M_p} _{M_p} (\mathbb{R}^d) = {\mathcal S}^{(M_p)}  (\mathbb{R}^d) $.
Moreover, the Fourier transform $\mathcal F$ extends to a homeomorphism on $(\cS^{\{M_p\}})' (\mathbb{R}^d)$
and on $(\cS^{(M_p)})' (\mathbb{R}^d)$ in a usual way.

\par

Next we discuss the important case when $ (M_p)_{p \in
\mN_0} $ and $ (N_q)_{q \in \mN_0} $ are hcosen to be the Gevrey sequences
$ M_p = p!^{r}, $ $ p\in \mN_{0}$  and $ N_q = q!^{s}, $  $ q\in
\mN_{0}$, for some $r,s\geq 0$, then we use the notation
$$ {\mathcal S}^{N_q} _{M_p} (\R^d) =
{\mathcal S}^{s} _{r} (\R^d) \;\;\; \text{and} \;\;\; \Sigma^{N_q} _{M_p} (\R^d) =  \Sigma^{s} _{r} (\R^d).
$$
If, in addition, $ s = r, $ then we put
$$ {\mathcal S}^{\{ s\}}  (\R^d) =
{\mathcal S}^{s} _{s} (\R^d)  \;\;\; \text{and} \;\;\; \Sigma^{(s)} (\R^d) =  \Sigma^{s} _{s} (\R^d).
$$

The choice of Gevrey sequences is the most often used choice in the literature since it
serves well in different contexts. For example, when discussing
nontriviality of  Gelfand-Shilov spaces we have the following:
\begin{enumerate}
\item the space $ {\mathcal S}^{s} _{r} (\R^d)$
is nontrivial if and only if $ s+r > 1 $, or $ s+r = 1$ and $sr>0$,
\item if $ s+r \geq 1 $ and $s<1$, then every  $ f \in {\mathcal S}^{s} _{r} (\R^d)$
can be extended to the complex domain as an entire function,
\item if $ s+r \geq 1 $ and $s=1$, then every  $ f \in {\mathcal S}^{s} _{r} (\R^d)$
can be extended to the complex domain as a holomorphic function in a strip
$ \{ x+iy \in \mC ^d \; : \; |y| < T \} $ some $ T>0$
\item the space $ \Sigma^{s} _{r} (\R^d)$
is nontrivial if and only if $ s+r > 1 $, or, if $ s+r = 1$ and $sr>0$ and $ (s,r)\neq (1/2,1/2)$.
\end{enumerate}

We refer to \cite{GS} or \cite{NR} for the proof in the case of $ {\mathcal S}^{s} _{r} (\R^d)$,
and to \cite{P1} for the spaces $  \Sigma^{s} _{r} (\R^d) $, see also \cite{Toft-2012}.

The discussion here above shows that Gelfand-Shilov classes $ {\mathcal S}^{s} _{r} (\R^d)$
consist of quasi-analytic functions when  $s \in (0,1)$.
This is in a sharp contrast with e.g. Gevrey classes
$  G^{s} (\mathbb{R}^d),$ $s>1$, another family of functions commonly used in regularity theory
of partial differential equations, whose elements are always non-quasi-analytic. Recall,
for $ 1 < s < \infty $ and an open set $ X \in \R^d$ the Gevrey class $ G^s (X) $ is given by
\begin{multline*}
G^s (X) = \{ \phi \in C^\infty(X) \;\; | \;\; (\forall K \Subset X ) (\exists C > 0) ( \exists h > 0 ) \\
\sup_{x \in K} \left | \partial^\alpha  \phi (x) \right| \leq C h^{|\alpha|} |\alpha| ! ^s
\}.
\end{multline*}
We refer to \cite{R} for microlocal analysis in Gervey classes and note that
$$ G^{s} _0  (\mathbb{R}^d) \hookrightarrow {\mathcal S}_s ^s (\mathbb{R}^d)
\hookrightarrow  G^{s} (\mathbb{R}^d), \;\;\;s > 1.
$$

\par

When the spaces are nontrivial we have the inclusions:
$$
\Sigma^{s} _{r}  (\R^d) \hookrightarrow {\mathcal S}^{s} _{r} (\R^d) \hookrightarrow {\mathcal S}  (\R^d),
$$
and $ {\mathcal S}  (\R^d) $ can be revealed as the limiting case of  spaces $ S^{s} _{r} (\R^d)$, i.e.
$$
{\mathcal S}  (\R^d) = {\mathcal S}^{\infty} _{\infty} (\R^d) = \lim_{s, r \rightarrow \infty} {\mathcal S}^{s} _{r} (\R^d),
$$
when the passage to the limit when $ s$ and $ r$ tend to infinity is interpreted correctly, see \cite[page 169]{GS}.

\begin{remark} \label{fine tuning}
Note that $  \Sigma^{1/2} _{1/2} (\mathbb{R}^d)= \{ 0 \}$ and $  \Sigma^{s} _{s} (\mathbb{R}^d)$ is
dense in the Schwartz space whenever $s>1/2$. One may consider a "fine tuning", that is the
spaces $\Sigma^{N_q} _{M_p} (\mathbb{R}^d)$ such that
$$
\{ 0 \} = \Sigma^{1/2} _{1/2} (\mathbb{R}^d) \hookrightarrow \Sigma^{N_q} _{M_p} (\mathbb{R}^d)\hookrightarrow
{\mathcal S}^{N_q} _{M_p} (\mathbb{R}^d) \hookrightarrow \Sigma^{s} _{s}(\mathbb{R}^d), \;\;\; s > 1/2,
$$
see also Remark \ref{fino tuniranje}.
\end{remark}

We refer to \cite{Toft-2017} where it is shown how to overcome the minimality condition
($ \Sigma^{1/2} _{1/2} (\R^d) = 0$) by transferring the estimates for
$ \| x^{\alpha} \partial^{\beta}  f  \|_{L^\infty} $ into the  estimates of the form
$ \|  H^N f  \|_{L^\infty}  \lesssim h^N (N!)^{2s},$ for some (for every ) $ h>0$,
where $H = |x|^{2} -  \Delta $ is the harmonic oscillator.

\par

We also mention that the Gelfand-Shilov space of analytic functions   $  {\mathcal S} ^{(1)} (\mathbb{R}^d ) := \Sigma^{1} _{1}  (\R^d)$
plays a prominent role in the theory since it
is isomorphic to  the Sato test function space for the space of Fourier hyperfunctions.
More precisely, if  $ f \in  {\mathcal S}^{( 1)} (\mathbb{R}^d) $ then it can be extended to a
holomorphic function $f(x+iy)$ in the strip
$ \{ x+iy \in \mC ^d \; : \; |y| < T \} $ for some $ T>0$.
According to Theorem \ref{GS-characterization}, we have
$$ f \in  {\mathcal S}^{( 1)} (\mathbb{R}^d) \Longleftrightarrow
\sup_{x\in \mathbb{R}^d } |f(x)  e^{h\cdot |x|}| < \infty \;
\; \text{and} \;
\sup_{\omega \in \mathbb{R}^d } | \hat f (\omega)  e^{h\cdot |\omega|} | < \infty, $$
for every $ h > 0.$ This representation is used to establish an isomorphism between its dual space
$ ({\mathcal S} ^{(1)})'   (\mathbb{R}^d ) $ and the space of Fourier hyperfunctions, see \cite{CCK1994} for details.

\par

Already in \cite{GS} it is shown that the Fourier transform is a topological isomorphism between
$ {\mathcal S}_r ^s  (\R^d)$ and $ {\mathcal S}_s ^r  (\R^d)$
$( {\mathcal F} ( {\mathcal S}_r ^s) = {\mathcal S}_s ^r)$,
which extends to a continuous linear transform
from $ ({\mathcal S}_r ^s) '  (\R^d)$ onto $ ({\mathcal S}_s ^r)'  (\R^d)$.
In particular, if $ s = r $ and $s \geq 1/2 $
then $ {\mathcal F} ( {\mathcal S}_s ^s) (\R^d)= {\mathcal S}_s ^s (\R^d),$
and $\cS^{1/2} _{1/2} (\mathbb{R}^d)$ is the smallest non-empty Gelfand-Shilov space invariant
under the Fourier transform, cf. \cite[Remark 1.2]{Teofanov2018}.
Similar assertions hold for $\Sigma^{s} _{r} (\R^d).$

\par

\subsection{Test function spaces on open sets}

Since we are interested in non-quasianalytic classes, we restrict our intention to the sequences which satisfy $(M.1) - (M.3)'$, and
refer to \cite{K1} for a more general setting.

\par

\begin{definition}
Let there be given a sequence $(M_p)$, $p\in \mathbb{N}^d,$
which satisfies $(M.1) - (M.3)'$ and let $ X $ be an open set in $ \mathbb{R}^d$.
For a given compact set $K \subset X$ and a constant $ A>0$ we denote by $ \mathcal{E}^{M_p} _{A,K} (X)$
the space of all  $\fy \in C^{\infty}(X)$ such that the norm
\begin{equation} \label{ultra-norm}
\| \fy \|_{M_p,A,K} = \sup_{p\in \mathbf N_0^n} \sup_{x\in K}
\frac{A^{|p|}}{M_p} |\fy^{(p)}(x)| <\infty.
\end{equation}
Note that $ \| \cdot \|_{M_p,A,K}$ is a norm in $ \mathcal{E}^{M_p} _{A,K} (X)$.

The space of functions $\fy \in C^{\infty}(X)$ such that
\eqref{ultra-norm} holds and $ \supp \fy \subseteq K$ is denoted
by $ \mathcal{D}^{M_p} _{A} (K) $.
\par

Let $(K_n)_n$ be a sequence of compact sets such that
$K_n\subset \subset K_{n+1}$ and $\bigcup K_n =X$. Then
\begin{align*}
\mathcal{E}^{(M_p)}(X) &= \operatorname{proj}\lim_{n\to \infty}
(\operatorname{proj}\lim_{A\to \infty} \mathcal{E}^{M_p} _{A,K_n})(X),
\\[1ex]
\mathcal{E}^{\{M_p\}}(X)&=\operatorname{proj}\lim_{n\to \infty}
(\operatorname{ind}\lim_{A\to 0} \mathcal{E}^{M_p} _{A,K_n})(X),
\\[1ex]
\mathcal{D}^{(M_p)}(X) &= \operatorname{ind}\lim_{n\to \infty}
(\operatorname{proj}\lim_{A\to \infty} \mathcal{D}^{M_p} _A (K_n))
\\[1ex]
 &= \operatorname{ind}\lim_{n\to \infty}
(\mathcal{D}^{(M_p)} _{K_n}),
\\[1ex]
\mathcal{D}^{\{M_p\}}(X)&=\operatorname{ind}\lim_{n\to \infty}
(\operatorname{ind}\lim_{A\to 0} \mathcal{D}^{M_p} _A (K_n))
\\[1ex]
 &= \operatorname{ind}\lim_{n\to \infty}
(\mathcal{D}^{\{M_p\}} _{K_n}).
\end{align*}
\end{definition}

\par

Obviously, $\mathcal{D}^{(M_p)}(X)$  ($\mathcal{D}^{\{M_p\}}(X)$  resp.) is the subspace of
$ \mathcal{E}^{(M_p)} (X) $ (of $ \mathcal{E}^{\{M_p\}} (X) $ resp.)
whose elements are compactly supported.

\par

\begin{remark}
Let $*$ denote $(M_p)$ or $\{ M_p \}$. Then $\mathcal{D}^{*}$, $\mathcal{S}^{*}$ and $\mathcal{E}^{*}$ correspond to $C_0^\infty$, $\mathcal{S}$ and $C^\infty$, respectively, and
$$
\mathcal{D}^{*}\subseteq C_0^\infty ,\quad \mathcal{S}^{*}\subseteq \mathcal{S}\quad \text{and}\quad
\mathcal{E}^{*}\subseteq C^\infty .
$$
\end{remark}

\par

The spaces of linear functionals over  $\mathcal{D}^{(M_p)}(X)$ and
$\mathcal{D}^{\{M_p\}}(X)$, denoted by $ (\mathcal{D}^{(M_p)})'(X)$
and $ (\mathcal{D}^{\{M_p\}})'(X)$ respectively, are called the
spaces of {\em ultradistributions} of Beurling and Roumieu  type
respectively, while the spaces of linear functionals over  $
\mathcal{E}^{(M_p)}(X)$ and    $ \mathcal{E}^{\{M_p\}}(X)$, denoted by
$ (\mathcal{E}^{(M_p)})'(X)$ and $ (\mathcal{E}^{\{M_p\}})'(X)$,
respectively are called the spaces of {\em ultradistributions
of compact support} of  Beurling  and Roumieu type respectively. Clearly,
\begin{align*}
(\mathcal{E}^{ \{M_p\} })'(X) &\subseteq (\mathcal{E}^{(M_p)})'(X),
\quad
 (\mathcal{E}^{(M_p)})'(X) \subseteq
(\mathcal{E}^{(M_p)})'(\mathbb{R}^d) \quad
\text{and}
\\[1ex]
(\mathcal{E}^{\{M_p\}})'(X) &\subseteq (\mathcal{E}^{\{M_p\}})'(\mathbb{R}^d).
\end{align*}
Moreover,
\begin{align*} (\mathcal{E}^{\{M_p\}})' (\mathbb{R}^d)  &\subseteq
(\cS^{\{M_p\}})' (\mathbb{R}^d)
\subseteq
(\mathcal{D}^{\{M_p\}})' (\mathbb{R}^d)
\\[1ex]
 \intertext{and}
 (\mathcal{E}^{(M_p)})'(\mathbb{R}^d) &\subseteq
(\cS^{(M_p)})' (\mathbb{R}^d)
\subseteq
(\mathcal{D}^{(M_p)})' (\mathbb{R}^d).
\end{align*}
Any ultra-distribution with compact support
can be viewed as an element of  $ (\cS^{(1)})' (\mathbb{R}^d)$. More generally, by using similar reasoning as in the case of distributions (see \cite{Ho1}), it follows that $\mathcal E^*$ are exactly those elements in $\mathcal S^*$ or $\mathcal D^*$ with compact support.

The following fact follows from the Paley-Wiener type theorems which can be found e.g. in \cite{K1}.

\begin{theorem} \label{Paley-Wiener theorem}
Let there be given a sequence $(M_p)$, $p\in \mathbb{N}^d,$
which satisfies $(M.1) - (M.3)'$ and let $ K $ be a compact convex set in $ \mathbb{R}^d$.
Then $ \varphi \in \mathcal{D}^{(M_p)} _{K} $ ($ \varphi \in \mathcal{D}^{\{M_p\}} _{K} $ resp.)
if and only if
for every $h>0$ there is  a constant $C>0$ (there are constants $h>0$ and $C>0$ resp.)
such that
$$
| \hat \varphi (\xi) | \leq C e^{-h M(|\xi|)}, \;\;\; \xi \in \mathbb{R}^d.
$$
\end{theorem}

%%%%%%%%%%%%%%%%%%%%%%%%%%%%%%%%%%%%%%%%%%%%%%%%%%
\section{Wave-front sets in weighted Fourier-Lebesgue spaces}\label{sec2}
%%%%%%%%%%%%%%%%%%%%%%%%%%%%%%%%%%%%%%%%%%%%%%%%%%

Although in principle both Beurling and Roumieu cases  could be treated simultaneously (as we did in Section \ref{sec1}),
in order to simplify the exposition, from now on we will treat the Beurling case only. See also \cite{JPTT3} for a discussion related to a slight difference between the cases.

Throughout the section $\{ M_p\}$ will always denote a sequence satisfies $ (M.1) - (M.3)'$ and $ M(\rho) $
denotes its associated function. For the notational convenience, the set of weights $\omega$  moderated with respect to the weight $ e^{M(\rho)} $ will be denoted by $ \mathcal{M}_{M(\rho)} (\mathbb{R}^{d})$ (instead of a more cumbersome notation $ \mathcal{M}_{e^{M(\rho)}} (\mathbb{R}^{d})$).

\par

Let $q\in [1,\infty ]$ and let $\omega \in \mathcal M_{M(\rho)} (\mathbb{R}^d)$.
The (weighted) Fourier Lebesgue space $\mathcal
FL^q_{(\omega )}(\mathbb{R}^d)$ is the inverse Fourier image of
$L^q _{(\omega )} (\mathbb{R}^d)$, i.{\,}e. $\mathcal FL^q_{(\omega )}(\mathbb{R}^d)$
consists of all $f\in (\mathcal S ^{(1)})'(\mathbb{R}^d)$ such that
\begin{equation*}
\nm f{\mathcal FL^{q}_{(\omega )}} \equiv \nm {\widehat f\cdot \omega }{L^q} .
\end{equation*}
is finite. If $\omega =1$, then the notation $\mathcal FL^q$
is used instead of $\mathcal FL^q_{(\omega )}$. We note that if
$\omega (\xi )=\eabs \xi ^s$, then $\mathcal
FL^{q}_{(\omega )}$ is the Fourier image of the Bessel potential space
$H^p_s$. % (cf. \cite{BL}).

\par

\begin{remark} \label{xdependence}
We may permit an $x$
dependency for the weight $\omega$ in the definition of Fourier
Lebesgue spaces. More precisely, for each $\omega \in \mathcal M_{M(\rho)} (\mathbb{R}^{2d})$
we let $\mathcal FL^q_{(\omega )}$ be the set of all ultradistributions $f$ such that
\[
\nm f{\mathcal FL^q_{(\omega )}}
%= \nm f{\mathcal FL^q_{(\omega),x}}
\equiv \nm {\widehat f\, \omega (x,\cdo )}{L^q}
\]
is finite. Since $\omega$ is $v_k$-moderate it follows that
different choices of $x$ give rise to equivalent norms, hence
$\nm f{\mathcal FL^{q}_{(\omega )}}<\infty$ is
independent of $x$. Therefore, a $\mathcal FL^q_{(\omega )}(\mathbb{R}^d)$ is  independent of $x$ although $\nm \cdo {\mathcal
FL^{q}_{(\omega )}}$ might depend on $x$.
\end{remark}

\par

Next we introduce local  Fourier-Lebesgue spaces of ultradistributions related to the given sequence  $\{ M_p\}$.
Let $X $ be an open set in $ \mathbb{R}^d$ and let $\omega \in \mathcal{M}_{M(\rho)} (\mathbb{R}^{d})$.
The \emph{local} Fourier Lebesgue space $\mathcal FL^q_{(\omega ),loc}(X)$ consists of all
$f\in (\mathcal S ^{(1)})'(\mathbb{R}^d)$ such that $\fy f\in \mathcal FL^q_{(\omega )}(\mathbb{R}^d)$ for each
$\fy \in \mathcal{D}^{(M_p)} (X)$. It is a Fr\'echet space under the topology given by the family of
seminorms $f\mapsto  \nm {\fy f}{\mathcal FL^q_{(\omega )}}$, where $\fy \in \mathcal{D}^{(M_p)} (X)$,
and the  following  simple properties hold.

\begin{lemma}
Let there be given a sequence  $\{ M_p\}$ with the associate function $M(\rho),$ $\rho >0.$
Let $X $ be an open set in $ \mathbb{R}^d$ and  $\omega \in \mathcal{M}_{M(\rho)} (\mathbb{R}^{d})$. Then
\begin{equation}\label{FLqFLqlocemb}
\mathcal FL^q_{(\omega )}(\mathbb{R}^d)  \subseteq \mathcal FL^q_{(\omega),loc}(\mathbb{R}^d)
\subseteq \mathcal FL ^q_{(\omega ),loc}(X).
\end{equation}
Furthermore, let $ q_1, q_2 \in [1, \infty] $ and $ \omega_1, \omega_2 \in \mathcal{M}_{M(\rho)}(\mathbb{R}^d)$.
Then
\begin{equation}\label{incrFLloc}
\mathcal FL^{q_1}_{(\omega _1),loc}(X)\subseteq \mathcal
F L^{q_2}_{(\omega _2),loc}(X), \;\; \text{when} \;\; q_1\le q_2 \;\;
\text{and} \;\; \omega _2\lesssim \omega _1.
\end{equation}
\end{lemma}

\begin{proof}
If $f\in \mathcal FL^q_{(\omega )}(\mathbb{R}^d)$ and if $\fy  \in
\mathcal{D}^{(M_p)} (X)$, then Young's inequality gives
\begin{multline*}
\nm {\fy f}{\mathcal FL^q_{(\omega)}} = \nm {\mathcal F(\fy f)\,
\omega}{L^q} =(2\pi )^{-d/2} \nm {(\widehat \fy *\widehat f \, )\,
\omega}{L^q}
\\[1ex]
\lesssim
\nm {|\widehat \fy \, e^{M(\cdot)} |* |\widehat f\, \omega |}{L^q} \lesssim \nm {\widehat f\, \omega}{L^q} = \nm {f}{\mathcal
FL^q_{(\omega)}},
\end{multline*}
if $ \nm {\widehat \fy \,e^M(\cdot) }{L^1}$ is finite.
Since $\fy \in \mathcal{D}^{(M_p)}(X),$ from Theorem \ref{Paley-Wiener theorem} and Remark \ref{observation} it follows that
for every $ N>0$ we have
\begin{equation}\label{psiexpest}
|\widehat \fy (\xi )e^{M(\xi)} | \lesssim e^{-(N+1)M(\xi)} e^{M(\xi)} = e^{-N M(\xi)}.
\end{equation}
Therefore $ \nm {\widehat \fy e^{M(\cdot)}}{L^p} < \infty$ for every $p\in [1,\infty ]$, and \eqref{FLqFLqlocemb} is proved.

\par

It remains to prove  \eqref{incrFLloc}.
The inclusion in \eqref{incrFLloc} is clear when $ q_1 = q_2 $ and
$\omega _2\lesssim\omega _1 $. It remains
to show that $\mathcal FL^{q}_{(\omega ),loc}$ increases with respect to $q$.
Assume, without any loss of generality, that
$f\in (\mathcal E^{(M_p)})' (X)$, and that $\fy \in \mathcal{D}^{(M_p)} (\mathbb{R}^d)$ is such that $\fy \equiv 1$
in the neighborhood of $\supp f$.
Choose $p\in [1,\infty ]$ such that
$1/q_1+1/p=1/q_2+1$. Then, for a $e^{M(\cdot)}$-moderate weight $\omega$, it follows from
Young's inequality that
$$
\nm {f}{\mathcal FL^{q_2}_{(\omega)}} \lesssim \nm {(\widehat \fy
*\widehat f\, ) \omega}{L^{q_2}}
\lesssim \nm {\widehat \fy e^{M(\cdot)} }{L^p}\nm {\widehat f\omega}{L^{q_1}} =C\nm {f}{\mathcal
FL^{q_1}_{(\omega)}},
$$
for some constant $C$, and the result follows.
\end{proof}

\par

Next we extend the definition of wave-front sets of Fourier-Lebesgue type  given in \cite{JPTT2, PTT3, PTT}.

\par

Let  $\{ M_p\}$ satisfy $ (M.1) - (M.3)'$ and let  $ M(\rho) $ denote its associated function. Furthermore, let
$ q\in [1,\infty ]$, and  $\Gamma \subseteq \mathbb{R}^d\setminus 0$ be an open cone.
If $f\in (\cS^{(1)})'(\mathbb{R}^d)$ and $\omega \in \mathcal M_{M(\rho)} (\mathbb{R}^{2d})$,
then we define
\begin{equation}
|f|_{\mathcal FL^{q,\Gamma}_{(\omega )}} = |f|_{\mathcal
FL^{q,\Gamma}_{(\omega ), x}}
\equiv
\Big ( \int _{\Gamma} |\widehat f(\xi )\omega (x,\xi )|^{q}\, d\xi
\Big )^{1/q}\label{skoff1}
\end{equation}
(with obvious interpretation when $q=\infty$). We
note that $|\cdo |_{\mathcal FL^{q,\Gamma}_{(\omega ), x}}$ defines
a semi-norm on $(\cS^{(1)})'(\mathbb{R}^d)$ which might attain the value $+\infty$.
Since $\omega $ is $M(\rho)$-moderate  it follows that different $x \in \mathbb{R}^d$ gives rise to
equivalent semi-norms $ |f|_{\mathcal FL^{q,\Gamma}_{(\omega ),
x}}$, see Remark \ref{xdependence}. Furthermore, if $\Gamma =\mathbb{R}^d\setminus 0$, $f\in
\mathcal FL^{q}_{(\omega )}(\mathbb{R}^d)$ and $q<\infty$, then
$|f|_{\mathcal FL^{q,\Gamma}_{(\omega ), x}}$ agrees with the Fourier
Lebesgue norm $\nm f{\mathcal FL^{q}_{(\omega ), x}}$ of $f$.

\par

For the sake of notational convenience we set
\begin{equation} \label{notconv1}
\mathcal B=\mathcal FL^q_{(\omega )}=\mathcal FL^q_{(\omega )} (\mathbb{R}^d), \quad
\mbox{and}
\quad
|\cdo |_{\mathcal B(\Gamma )}=|\cdo |_{\mathcal
FL^{q,\Gamma}_{(\omega ), x}}.
\end{equation}
We let $\Theta _{\mathcal B}(f)=\Theta _{\mathcal FL^{q} _{(\omega
)}} (f)$ be the set of all $ \xi \in \mathbb{R}^d\setminus 0 $ such that
$|f|_{\mathcal B(\Gamma )} < \infty$, for some open conical
neighborhood $\Gamma = \Gamma_{\xi}$ of $\xi$.  We also let
$\Sigma_{\mathcal B} (f)$ be the complement of $ \Theta_{\mathcal
B} (f)$ in $\mathbb{R}^d\setminus 0 $. Then
$\Theta_{ \mathcal B} (f)$ and $\Sigma_{\mathcal
B} (f)$ are open respectively
closed subsets in $\mathbb{R}^d\setminus 0$, which are independent of
the choice of $ x \in \mathbb{R}^d$ in \eqref{skoff1}.

\par

\begin{definition}\label{wave-frontdef1}
Let there be given a sequence $\{ M_p\}$ which satisfies $ (M.1) - (M.3)'$ and let $ M(\rho) $ be its associated function.
Furthermore, let $q\in [1,\infty ]$, $\mathcal B$ be as in \eqref{notconv1}, and let
$X$ be an open subset of $\mathbb{R}^d$. If
$\omega \in \mathcal{M}_{M(\rho)} (\mathbb{R}^{2d})$,
then the wave-front set of
$f\in (\cD^{*})' (X)$,
$
\WF _{\mathcal B}(f)  \equiv  \WF _{\mathcal FL^q_{(\omega )}}(f)
$
with respect to $\mathcal B$ consists of all pairs $(x_0,\xi_0)$ in
$X\times (\mathbb{R}^d \setminus 0)$ such that
$
\xi _0 \in  \Sigma _{\mathcal B} (\fy f)
$
holds for each $\fy \in \mathcal{D}^{(M_p)} (X)$ such that $\fy (x_0)\neq 0$.
\end{definition}

\par

The set $\WF  _{\mathcal B}(f)$ is a closed set in $\mathbb{R}^d\times
(\mathbb{R}^d\setminus 0)$, since it is obvious that its complement is
open. We also note that if $ x\in \mathbb{R}^d$ is fixed and $\omega _0(\xi
)=\omega (x,\xi )$, then $\WF _{\mathcal B} (f)=\WF _{\mathcal
FL^q_{(\omega _0)}}(f)$, since $\Sigma _{\mathcal B}$ is independent
of $x$.

\par

The following theorem shows that wave-front sets with respect to
$\mathcal FL^q_{(\omega )}$ satisfy appropriate micro-local
properties. It also shows that such wave-front sets are decreasing
with respect to the parameter $q$, and increasing with respect to the
weight $\omega$.

\par

\begin{theorem}\label{wavefrontprop11}
Let there be given a sequence $\{ M_p\}$ which satisfy $ (M.1) - (M.3)'$ and let $ M(\rho) $ be its associated function.
Furthermore, let $q,r\in [1,\infty ]$, $X$ be an open set in $\mathbb{R}^d$ and
$\omega ,\vartheta \in \mathcal{M}_{M(\rho)} (\mathbb{R}^{2d})$
be such that
\begin{equation*}
r\le q,\quad \text{and}\quad \omega
(x,\xi )\lesssim \vartheta (x,\xi ).
\end{equation*}
Also let  $\mathcal B$  be as in \eqref{notconv1} and put
$ {\mathcal B _0} =\mathcal FL^r_{(\vartheta )}
(\mathbb{R}^d)$. If $f\in (\cD^{(M_p)})' (X)$ and  $\fy \in \mathcal{D}^{(M_p)}(X)$
then
\[
\WF _{ {\mathcal B} }(\fy \, f)\subseteq \WF _{\mathcal B_0}(f).
\]
\end{theorem}

\par

\begin{proof}
When $M_p = p!^s$, $s>1$, we recover \cite[Theorem 2.1]{JPTT2}. 
In fact, the more general situation when $\{ M_p\}$ is an arbitrary sequence which satisfies $ (M.1) - (M.3)'$ 
can be proved by using the idea of the proof of \cite[Theorem 2.1]{JPTT2} as follows.

By the definition it is sufficient to prove
\begin{equation*}
\Sigma_{ {\mathcal B } } (\fy f) \subseteq
\Sigma_{\mathcal B_0} (f)
\end{equation*}
when  $ \fy \in  \mathcal{D}^{(M_p)}(X)$, $\vartheta=\omega$ and $f\in (\mathcal E^{(M_p)})'(\rr d)$,
since the statement only involves local assertions. For the
same reasons we may assume that $\omega (x,\xi ) =\omega (\xi )$ is independent of $x$.
We prove the assertion for $ r \in [1, \infty )$, and leave the case $ r = \infty $  to the reader.

\par

By using the idea of the proof of  \cite[Theorem 1.6.1]{R} we conclude that
if $f\in (\mathcal E^{(M_p)})'(\mathbb{R}^d)$ then
there exists $N_0>0$ such that $|\widehat f(\xi )\omega(\xi )|\lesssim e^{N_0 M(|\xi|)}$.

Choose open cones $\Gamma _1$ and $\Gamma_2$ in $\mathbb{R}^d$ such that
$\overline {\Gamma _2}\subseteq \Gamma _1$.
It is enough to prove that
for every $N>0$, there exist $C_N>0$ such that
\begin{equation} \label{cuttoff1}
|\fy f|_{ {\mathcal B} (\Gamma _2)}\le C_N \Big (|f|_{\mathcal B_0(\Gamma _1)} +
\sup _{\xi \in \mathbb{R}^d} \big ( |\widehat f(\xi )\omega(\xi )|e^{-N M(|\xi|)} \big )
\Big )
\end{equation}
when $ \overline \Gamma _2\subseteq \Gamma_1.$

\par

Since $\omega \in \mathcal{M}_{M(\rho)}  (\mathbb{R}^{d})$ by letting $F(\xi )=|\widehat f(\xi )\omega (\xi ) |$
and $\psi (\xi )=|\widehat \fy (\xi )| e^{M(|\xi| )}$ we have
\begin{multline*}
|\fy f| _{ {\mathcal B}(\Gamma _2)} = \Big (\int
_{\Gamma _2}|\mathcal F(\fy f)(\xi )\omega(\xi )|^{q}\, d\xi
\Big )^{1/q}
\\[1ex]
\lesssim\Big (\int _{\Gamma _2}\Big ( \int _{\rr {d}} \psi (\xi -\eta
)F(\eta )\, d\eta \Big )^{q}\, d\xi \Big )^{1/q} \lesssim J_1+J_2,
\end{multline*}
where
\begin{align*}
&J_1  = \Big (\int _{\Gamma _2}\Big ( \int _{\Gamma _1}\psi (\xi
-\eta )F(\eta )\, d\eta \Big )^{q}\, d\xi \Big )^{1/q},
\\[1ex]
&J_2 = \Big (\int _{\Gamma _2}\Big ( \int _{\complement \Gamma _1}\psi
(\xi -\eta )F(\eta )\, d\eta \Big )^{q}\, d\xi \Big )^{1/q}.
\end{align*}
\par

Let $q_0$ be chosen such that $1/r_0+1/r=1+1/q$, and let $\chi
_{\Gamma _1}$ be the characteristic function of $\Gamma _1$. Then
Young's inequality gives
\begin{multline*}
J_1 \le  \Big (\int _{\rr {d}} \Big ( \int _{\Gamma _1}\psi (\xi
-\eta )F(\eta )\, d\eta \Big )^{q}\, d\xi \Big )^{1/q}
\\[1ex]
=\nm {\psi *(\chi _{\Gamma _1} F)}{L^{q}} \le \nm \psi {L^{r_0}}\nm
{\chi _{\Gamma _1} F}{L^{r}} = C_\psi |f|_{\mathcal B_0(\Gamma _1)},
\end{multline*}
where $C_\psi = \nm \psi {L^{q_0}}<\infty$.

To estimate $J_2$, we note that since $\fy \in \mathcal{D}^{(M_p)}(X)$, then by Theorem \ref{Paley-Wiener theorem} it follows that
for every $ N>0$ there exist $ C_N>0 $ such that
\begin{equation}\label{psiexpest}
\psi (\xi) = |\widehat \fy (\xi )e^{M(|\xi| )} \leq  C_{{N}} e^{-(N+1)M(|\xi|)} e^{M(|\xi|)} \leq C_{{N}} e^{-N M(|\xi|)}.
\end{equation}

Furthermore,  $\overline {\Gamma _2} \subseteq \Gamma _1$ implies that
\begin{multline}\label{estimate1}
|\xi -\eta |>2c\max
(|\xi|,|\eta |)
\\[1ex]
\geq c (|\xi|+|\eta|),\qquad \xi \in
\Gamma _2,\  \eta \notin \Gamma _1
\end{multline}
holds for some constant $c>0$, since this is true when
$1=|\xi |\ge |\eta|$.
Now, a combination of  Lemma \ref{osobineasocirane}, \eqref{psiexpest} and \eqref{estimate1} 
(together with the monotone increasing property of $ M(\rho)$) implies that for every $N_1>0$ we have
\[
\psi(\xi-\eta) \lesssim C e^{-2N_1(M(|\xi|) +M(|\eta|))},
\]
%%
%It follows that for some $0<N<(\tilde{N} -\ep)c/2-d$
which gives
\begin{multline*}
J_2 \lesssim \Big (\int _{\Gamma _2}\Big ( \int _{\complement
\Gamma _1}e^{-2N_1(M(|\xi|) +M(|\eta|))} F(\eta )\, d\eta \Big
)^{r}\, d\xi \Big )^{1/r}
\\[1ex]
\lesssim \Big (\int _{\Gamma _2}\Big ( \int _{\complement \Gamma_1}
e^{-2N_1(M(|\xi|) +M(|\eta|))} e^{N_1 M(|\eta|)}(e^{-N_1 M(|\eta|)}
F(\eta))\, d\eta \Big )^{r}\, d\xi \Big )^{1/r}
\\[1ex]
\lesssim \sup _{\eta \in \mathbb{R}^d}|e^{-N_1M(|\eta|)}  F(\eta ))|.
\end{multline*}
This implies \eqref{cuttoff1} and the proof is finished.
\end{proof}

%%%%%%%%%%%%%%%%%%%%%%%%%%%%%%%%%%%%
\subsection{Comparisons to other types of wave-front sets} \label{subsec1.2}
%%%%%%%%%%%%%%%%%%%%%%%%%%%%%%%%%%%%

\par

Let $\omega \in \mathcal{M}_v (\mathbb{R}^{2d})$ be moderated with
respect to the weight $v$ of a polynomial growth at infinity, and let $f
\in \mathcal{D}' (X). $ Then the wave frpont set $ \WF _{\mathcal FL^q_{(\omega
)}}(f)$ in Definition \ref{wave-frontdef1} agrees with the
wave-front set introduced in \cite[Definition 3.1]{PTT}.
Therefore, the information on regularity in the background of
wave-front sets of Fourier-Lebesgue type in Definition
\ref{wave-frontdef1} might be compared to the information
obtained from the classical wave-front sets, cf. Example 4.9 in
\cite{PTT}.

\par

Next we compare the wave-front sets introduced in Definition
\ref{wave-frontdef1} to the wave-front sets in spaces of
ultradistributions given in \cite{Ho1, P, R}.

\par

Let $ s>1$ and let $X$ be an open subset of $\mathbb{R}^d$. The ultradistribution
$f \in (\mathcal{D}^{(s)})' (X)$ ($f \in (\mathcal{D}^{\{ s \} })' (X)$)
is $(s)$-micro-regular ($\{ s \}$-micro-regular) at $ (x_0, \xi_0)$ if there exists
$\fy \in \mathcal{D}^{(s)} (X)$ ($\fy \in \mathcal{D}^{\{ s \} } (X)$) such that $\fy (x) = 1$
in a neighborhood of $x_0$ and an open cone
$\Gamma $ which contains $\xi_0$ such that
\begin{equation} \label{Rodino wf set}
|\mathcal F(\fy f)(\xi )| \lesssim  e^{-N|\xi|^{1/s}}, \quad \xi \in \Gamma ,
\end{equation}
for each $N>0$ (for some $N>0$). The $(s)$-wave-front set ($\{ s \}$-wave-front set)
of $f$, $ \WF_{(s)} (f) $ ($ \WF_{\{ s \}} (f) $) is defined
as the complement in $X \times \mathbb{R}^d \setminus 0 $ of the set of all $ (x_0, \xi_0) $ where $f$ is
$(s)$-micro-regular ($\{ s \} $-micro-regular), cf. \cite[Definition 1.7.1]{R}.

\par

The $\{ s\}$-wave-front set  $ \WF_{\{ s\} } (f) $ can be found in \cite{P} and it coincides
with certain wave-front set $ \WF_L (f) $ introduced in \cite[Chapter 8.4]{Ho1}.

\par

Next  we modify the definitions from \cite{PTT, JPTT2}.

\par

Let there be given a sequence $\{ M_p\}$ which satisfy $ (M.1) - (M.3)'$ and let $ M(\rho) $ be its associated function.
Furthermore, let $\omega _j \in \mathcal M_{M(\rho)}(\mathbb{R}^{2d})$, $q_j\in [1,\infty ]$
when $j$ belongs to some index set $J$, and let $\mathcal B$ be the array of spaces, given by
\begin{equation}\label{notconvsequences}
 (\mathcal B_j) \equiv (\mathcal B_j)_{j\in J},\quad
\text{where}\quad \mathcal B_j=\mathcal FL^{q_j}_{(\omega
_j)}=\mathcal FL^{q_j}_{(\omega _j)} (\mathbb{R}^d), \quad j \in J.
\end{equation}

\par

If  $f\in (\mathcal D^{(M_p)})'(\mathbb{R}^d)$, and $(\mathcal B_j)$ is given by
\eqref{notconvsequences}, then we let
$\Theta_{(\mathcal B_j) }^{\sup}(f)$ be the set of all $\xi \in \mathbb{R}^d\setminus 0$ such that for some $\Gamma = \Gamma _{\xi}$ and each
$j\in J$ it holds $|f|_{\mathcal B_j(\Gamma )} < \infty$. We
also let $\Theta _{(\mathcal B_j) }^{\inf}(f)$ be the set of all $
\xi \in \mathbb{R}^d\setminus 0 $ such that for some $\Gamma = \Gamma
_{\xi}$ and some $j\in J$ it holds $|f|_{\mathcal B_j(\Gamma)} <
\infty$. Finally we let $\Sigma _{(\mathcal B_j) }^{\sup} (f)$ and
$\Sigma _{(\mathcal B_j) }^{\inf} (f)$ be the complements in $\mathbb{R}^d\setminus 0 $ of $\Theta_{(\mathcal B_j) }^{\sup}(f)$ and $\Theta
_{(\mathcal B_j) }^{\inf} (f)$ respectively.

\par

\begin{definition}\label{defsuperposWF}
Let there be given a sequence $\{ M_p\}$ which satisfy $ (M.1) - (M.3)'$ and let $ M(\rho) $ be its associated function.
Furthermore, let $J$ be an index set, $q_j\in [1,\infty ]$, $\omega _j\in \mathcal
M_{M(\rho)}(\mathbb{R}^{2d})$ when $j\in J$, $(\mathcal B_j)$ be as in
\eqref{notconvsequences}, and let $X$ be an open subset of $\mathbb{R}^d$.
\begin{enumerate}
\item The wave-front set of $f\in (\mathcal D^{(M_p)})'(X)$,
of \emph{sup-type} with respect to $(\mathcal B_j) $,
$\WF ^{\, \sup} _{(\mathcal B_j) }(f) $,
 consists of all pairs
$(x_0,\xi_0)$ in $X\times (\mathbb{R}^d \setminus 0)$ such that
$
\xi _0 \in  \Sigma ^{\sup} _{(\mathcal B_j)} (\fy f)
$
holds for each $\fy \in \mathcal D^{(M_p)} (X)$ such that $\fy (x_0)\neq
0$;

\vrum

\item The wave-front set of $f\in (\mathcal D^{(M_p)})'(X)$,
of \emph{inf-type} with respect to $(\mathcal B_j)$,
$
\WF ^{\, \inf} _{(\mathcal B_j)}(f) $
consists of all pairs
$(x_0,\xi_0)$ in $X\times (\mathbb{R}^d \setminus 0)$ such that
$
\xi _0 \in  \Sigma ^{\inf} _{(\mathcal B_j)} (\fy f)
$
holds for each $\fy \in \mathcal D^{(M_p)} (X)$ such that $\fy (x_0)\neq
0$.
\end{enumerate}
\end{definition}

\par

Now we are ready to rewrite the classical Gevrey wave-front sets $ \WF_{\{ s\} } (f)$ and
$ \WF_{( s)} (f)$ in terms of  wave-front sets introduced in Definition \ref{defsuperposWF}.

\begin{proposition} \label{s-WF set} \cite{JPTT2}
Let  $ s>1$, and let $\mathcal B_j$ be the same as in \eqref{notconvsequences}
with $q _j\in [1,\infty]$ and $\omega _j(\xi)  \equiv e^{j |\xi|^{1/s}} $. Then the following is true:
\begin{enumerate}
\item if $f \in (\mathcal{D}^{\{ s \} })' (\mathbb{R}^d)$, then
$$
\WF _{(\mathcal B_j)} ^{\, \inf}(f) = \bigcap_{j>0} \WF _{\mathcal B_j} (f) =  \WF_{\{ s\} } (f)
\subseteq \WF_{(s)} (f) \text ;
$$

\vrum

\item if $f \in (\mathcal{D}^{(s)})' (\mathbb{R}^d)$, then
$$
\WF_{(s)} (f) =\bigcup_{j>0}
\WF _{\mathcal B_j} (f) \subseteq \WF _{(\mathcal B_j)} ^{\, \sup}(f).
$$
\end{enumerate}
\end{proposition}

\begin{remark}
We recall that if $f\in \mathcal D'(\mathbb{R}^d)$, and $\omega _j(x,\xi ) = \eabs \xi ^{j}$ for $j\in J=\mathbf
N$, then it follows that $\WF _{(\mathcal B_j)}^{\, \sup}(f)$ in
Definition \ref{defsuperposWF} is equal to the standard wave front
set $\WF (f)$ in Chapter VIII in \cite{Ho1}.
\end{remark}

\par

\subsection{Convolution}

We finish the section by recalling that the convolution properties, valid for
standard wave-front sets of H{\"o}rmander type, also hold for the
wave-front sets of Fourier Lebesgue types, see \cite{PTT2, PTT4} for related results in the framework of
tempered distributions.
More generally, the following convolution result holds true.

\par

\begin{theorem} \label{convwavefront}
Let there be given a sequence $\{ M_p\}$ which satisfy $ (M.1) - (M.3)'$ and let $ M(\rho) $ be its associated function.
Furthermore, let $q,q_1,q_2\in [1,\infty ]$ and  let $ \omega, \omega_1,
\omega_2 \in \mathcal{M}_{M(\rho)}  (\mathbb{R}^d)$ satisfy
\begin{equation}\label{qomegarel1}
\frac 1{q_1}+\frac 1{q_2} = \frac 1{q}\quad \text{and}\quad  \omega
(\xi) \lesssim \omega_1 (\xi) \omega_2 (\xi).
\end{equation}
%%
%Furthermore, let $(*)$ denote $(s)$ or $\{ s \}$.
Then the convolution map $(f_1,f_2)\mapsto f_1*f_2$ from
$\cS^{ (1) } (\mathbb{R}^d)\times \cS^{ (1) }(\mathbb{R}^d)$ to $ \cS^{ (1) } (\mathbb{R}^d)$
extends to a
continuous mapping from $\mathcal FL^{q_1}_{(\omega_1)}(\mathbb{R}^d)\times
\mathcal FL^{q_2}_{(\omega_2)}(\mathbb{R}^d)$ to $\mathcal
FL^{q}_{(\omega)}(\mathbb{R}^d)$. This extension is unique if $q_1<\infty$ or
$q_2<\infty$.

If
$f_1\in \mathcal FL^{q_1}_{(\omega _1),loc}(\mathbb{R}^d)$, $f_2\in (\mathcal{D}^{(M_p)})'
(\mathbb{R}^d)$ and $f_1$ or $f_2$ have compact supports, then
$$
\WF _{\mathcal FL^{q}_{(\omega )}}(f_1*f_2) \subseteq \sets {(x+y,\xi
)}{x\in \supp f_1\ \text{and}\ (y,\xi )\in \WF _{\mathcal
FL^{q_2}_{(\omega _2)}}(f_2)}.
$$
\end{theorem}

\par

The proof is omitted, since the arguments for the first part of Theorem are the same as in the proof of  \cite[Lemma 2.1]{PTT2},
taking into account that $\cS^{(1)}$ is dense in $\mathcal F L^q_{(\omega)}$ when $q<\infty$.
The second part of Theorem \ref{convwavefront} can be proved in the same way as \cite[Theorem 2.2]{JPTT3}.

\par

\section{Modulation Spaces} \label{Modulation Spaces} \label{sec3}

In this section we first recall the action of the short-time Fourier transform on
Gelfand-Shilov spaces and their dual spaces, and then proceed with modulation spaces and their properties.
Since the short-time Fourier transform gives a phase-space description of a function or distribution, we first
extend Definition \ref{GSoftypeS}.

%%%%%%%%%%%%%%%%%%%

\begin{definition} \label{GSoftypeS-2d}
Let there be given sequences of positive numbers
$ (M_p)_{p \in \mN_0} $, $ (N_q)_{q \in \mN_0} $,
$ (\tilde M_p)_{p \in \mN_0} $, $ (\tilde N_q)_{q \in \mN_0} $
which satisfy $ (M.1) $ and $ (M.2).$ We define
$ \cS^{N_{q} , \tilde N_{q} ,B} _{M_{p}, \tilde M_{p}, A} (\R^{2d}) $ to be the set of smooth functions
$ f \in C^\infty(\R^{2d}) $ such that
$$
\| x^{\alpha_1} \o^{\alpha_2}\partial^{\beta_1} _x \partial^{\beta_2} _\o f  \|_{L^\infty} \leq
C A^{|\alpha_1 + \alpha_2|} M_{|\alpha_1|} \tilde M_{|\alpha_2|}
B^{|\beta_1 + \beta_2|} N_{|\beta_1|} \tilde N_{|\beta_2|},
$$
$$
\forall \alpha_1, \alpha_2, \beta_1, \beta_2 \in \N_{0} ^d \},
$$
and for some $A, B, C >0.$ {\em Gelfand-Shilov spaces} are projective and inductive limits of
$ \cS^{N_{q} , \tilde N_{q} ,B} _{M_{p}, \tilde M_{p}, A} (\R^{2d}) $:
$$
\Sigma^{N_q, \tilde N_q} _{M_p, \tilde M_p} (\R^{2d}) : = {\rm proj} \lim_{A>0, B>0}
\cS^{N_{q} , \tilde N_{q} ,B} _{M_{p}, \tilde M_{p}, A} (\R^{2d});
$$
$$
\cS^{N_q, \tilde N_q} _{M_p, \tilde M_p} (\R^{2d})  : = {\rm ind} \lim_{A>0, B>0}
\cS^{N_{q} , \tilde N_{q} ,B} _{M_{p}, \tilde M_{p}, A} (\R^{2d}).
$$
\end{definition}

Clearly, the corresponding dual spaces are given by
$$
(\Sigma^{N_q, \tilde N_q} _{M_p, \tilde M_p})' (\R^{2d}) : = {\rm ind} \lim_{A>0, B>0}
(\cS^{N_{q} , \tilde N_{q} ,B} _{M_{p}, \tilde M_{p}, A})' (\R^{2d});
$$
$$
(\cS^{N_q, \tilde N_q} _{M_p, \tilde M_p})' (\R^{2d})  : = {\rm proj} \lim_{A>0, B>0}
(\cS^{N_{q} , \tilde N_{q} ,B} _{M_{p}, \tilde M_{p}, A})' (\R^{2d}).
$$

By Theorem \ref{GS-characterization},  the Fourier transform is a homeomorphism from
$ \Sigma^{N_q, \tilde N_q} _{M_p, \tilde M_p} (\R^{2d}) $
to $\Sigma_{N_q, \tilde N_q} ^{M_p, \tilde M_p} (\R^{2d}) $
and, if $\Fur_1 f$ denotes the partial Fourier transform of $f(x,\o)$
with respect to the $x$ variable, and if $\Fur_2 f$
denotes the partial Fourier transform of $f(x,\o)$
with respect to the $\o$ variable, then
$ \Fur_1 $ and $ \Fur_2$ are homeomorphisms from
$ \Sigma^{N_q, \tilde N_q} _{M_p, \tilde M_p} (\R^{2d}) $
to $\Sigma_{M_p, \tilde N_q} ^{N_q, \tilde M_p} (\R^{2d}) $ and
$ \Sigma^{N_q, \tilde M_p} _{M_p, \tilde N_q} (\R^{2d}) $, respectively.
Similar facts hold when $ \Sigma^{N_q, \tilde N_q} _{M_p, \tilde M_p} (\R^{2d}) $ is replaced by
$\cS^{N_q, \tilde N_q} _{M_p, \tilde M_p} (\R^{2d})$, $
(\Sigma^{N_q, \tilde N_q} _{M_p, \tilde M_p})' (\R^{2d})$ or
$ (\cS^{N_q, \tilde N_q} _{M_p, \tilde M_p})' (\R^{2d})$.

When $M_p = \tilde M_p $ and $ N_q = \tilde N_q $ we use usual abbreviated notation:
$\cS^{N_q} _{M_p}  (\R^{2d}) = \cS^{N_q, \tilde N_q} _{M_p, \tilde M_p} (\R^{2d})$
and similarly for other spaces.

\par

\subsection{Short-time Fourier transform}

Let $ (M_p)_{p \in \mN_0} $ satisfy $ (M.1) $ and  $ (M.2)$.
For any given  $f, g \in \cS ^{M_p}_{M_p} (\mathbb{R}^d )$
($f, g \in \Sigma ^{M_p}_{M_p} (\mathbb{R}^d )$, respectively) the \stft\ (STFT) of
$f $ with respect to the window $g$ is given by
\begin{equation*}
   V_g f(x,\xi)=
 (2\pi)^{-d/2} \int_{\mathbb{R}^d} f(y)\, {\overline {g(y-x)}} \, e^{-i\eabs{\xi,y}}\,dy\, .
 \end{equation*}

The following theorem (and its variations) is a folklore,
in particular in the framework of the duality between $\cS (\R^{2d})$ and $\cS^{'} (\R^{2d})$.
For Gelfand-Shilov spaces we refer to e.g.  \cite{GZ, T2, Teof2015, Toft-2012}.

\begin{theorem} \label{nec-suf-cond}
Let there be given  sequences $ ( M_p )_{p \in \mN_0} $ and
$ ( N_q )_{q \in \mN_0} $ which satisfy
(M.1), (M.2) and
$$
\{N.1\}: \;\;\;
(\exists H>0) (\exists A>0) \;\; p!^{1/2} \leq A H^p  M_{p}, \;\; p \in \mN_0.
$$
If $f,g \in \cS^{N_q} _{M_p} (\R^d)$,
then $ V_\phi f   \in \cS^{N_q, M_p} _{M_p, N_q} (\mathbb{R}^dd) $
and extends uniquely to a continuous map from
$ (\cS^{N_q} _{M_p})' (\R^d)\times  (\cS^{M_p} _{N_q})' (\R^d) $
into $ (\cS^{N_q,M_p} _{M_p, N_q})' (\R^{2d}).$

Conversely, if  $V_\phi f \in \cS^{N_q,M_p} _{M_p,N_q} (\mathbb{R}^dd) $
then $f,g \in \cS^{N_q} _{M_p} (\R^d).$

Next, assume that $ ( M_p )_{p \in \mN_0} $ and
$ ( N_q )_{q \in \mN_0} $ satisfy
(M.1), (M.2) and
$$
(N.1): \;\;
(\forall H>0) (\exists A>0) \;\; p!^{1/2} \leq A H^p  M_{p}, \;\; p \in \mN_0.
$$

If $f,g \in \Sigma^{N_q} _{M_p} (\R^d), $
then $ V_\phi f   \in \Sigma^{N_q, M_p} _{M_p, N_q}  (\mathbb{R}^d) $
and extends  uniquely to a continuous  map from $ (\Sigma^{N_q} _{M_p})' (\R^d) \times  (\Sigma^{M_p} _{N_q})' (\R^d) $
into $ (\Sigma^{N_q,M_p} _{M_p, N_q})' (\R^{2d}).$

Conversely, if  $ V_\phi f \in \Sigma^{N_q,M_p} _{M_p, N_q} (\mathbb{R}^dd)$
then $f,g \in \Sigma^{N_q} _{M_p} (\R^d).$
\end{theorem}

The  conditions $\{N.1\} $ and $(N.1)$ are taken from  \cite{LCP} where
they are called  {\em nontriviality conditions} for the spaces
${\mathcal S}^{M_p} _{M_p} (\R^{d})$ and $\Sigma^{M_p} _{M_p} (\R^{d})$ respectively, see also \cite{L06}.

We will also need the following proposition when proving that the wave-front sets of
Fourier-Lebesgue and modulation space types are the same. The first part is an
extension of \cite[Proposition 4.2]{CPRT1}.

\par

\begin{proposition} \label{stftestimates}
Let $\{ M_p\}$ satisfies $ (M.1) - (M.3)'$ and let $ M(\rho) $ denotes
its associated function. Then the following is true:

\begin{enumerate}
\item if $ f \in (\mathcal E ^{(M_p)})' (\mathbb{R}^d)$ and $\phi \in \cS^{(M_p)}(\mathbb{R}^d)$, then
\begin{equation}\label{rastiopadanje}
 |V_\phi f(x,\xi)|\lesssim e^{- hM(|x|)} e^{\ep M( | \xi |)} ,
 \end{equation}
for some $\ep >0$ and for every $h>0$;

\vrum
%
%\item if $ f \in (\mathcal E ^{\{ M_p \} })' (\mathbb{R}^d)$ and $\phi \in \cS^{(M_p)}(\mathbb{R}^d)$, then
%\eqref{rastiopadanje} holds for every $h>0$ and $\ep >0$;
%
%
%\vrum

\item if $ f \in (\mathcal D ^{ ( M_p ) })' (\mathbb{R}^d)$ and $\phi \in \cD^{(M_p)}(\mathbb{R}^d)\setminus 0$,
then $f\in (\mathcal E ^{ ( M_p ) })' (\mathbb{R}^d)$, if and only if $\supp V_\phi f\subseteq K\times \mathbb{R}^d$ for
some compact set $K$, and then
\begin{equation}\label{rastiopadanje2}
|V_\phi f(x,\xi)|\lesssim e^{\ep M(|\xi|)},
\end{equation}
for some $\ep >0$.
%
%\vrum
%
%\item if $ f \in (\mathcal D ^{(M_p)})' (\mathbb{R}^d)$ and $\phi \in \cD^{(M_p)}(\mathbb{R}^d)\setminus 0$,
%then $f\in (\mathcal E ^{ \{ M_p \} })' (\mathbb{R}^d)$, if and only if $\supp V_\phi f\subseteq K\times \mathbb{R}^d$ for
%some compact set $K$ and \eqref{rastiopadanje2} holds for every $\ep >0$.
\end{enumerate}
\end{proposition}

\begin{proof}
We only prove (1) and (3). The other statements follow by similar arguments
and are left for the reader. As before, we will use Remark \ref{observation} in our calculations.
Recall,  $ f \in (\mathcal E ^{(M_p)})' (\rr d)$ implies that
\[|\widehat{f}(\xi)|\lesssim e^{\ep M( |\xi|)},
\]
for some $\ep > 0,$ cf. \cite[Theorem 1.6.1]{R}.

For $\phi \in \cS^{(M_p)}(\rr d)$ and
$\psi\in \cD^{(M_p)}(\rr d)$ such that  $\psi=1$ in $\supp f$ by Theorem \ref{nec-suf-cond}, Lemma \ref{osobineasocirane}
and Remark \ref{observation}
it follows that
\[
|V_\psi \phi(x,\xi)| \lesssim e^{-h M(|x|)- k M( |\xi |)},
\]
for every $ h,k >0$. Now straight-forward calculations give
\begin{multline*}
|V_{\phi} f(x,\xi)| = |(V_{\phi} (\psi f))(x,\xi)| \lesssim (|V_{\psi}\phi(x,\cdot)|* |\widehat{f}|)(\xi)
\\[1ex]
= \int |V_{\psi}\phi (x,\xi - \eta )||
\widehat{f}(\eta )|\, d\eta \lesssim  \int e^{-h M(|x|)-2\ep M( |\xi-\eta |)} e^{\ep M(|\eta |)}\, d\eta
\\[1ex]
\leq  e^{-h M( |x|)} \int e^{-2\ep M( |\eta |)+2\ep M( |\xi|)+\ep M( |\eta |)}\, d\eta
\lesssim
e^{-h M( |x|)+2\ep M( |\xi|)},
\end{multline*}
and (1) follows.

\par

Next we prove (3). First assume that $\phi \in \cD^{(M_p)}(\rr d)\setminus 0$ and $f\in (\mathcal
E^{(M_p)})'(\rr d)$. Since both $\phi$ and $f$ have compact support, it follows that
$\supp (V_\phi f)\subseteq K\times \rr d$. Furthermore, by slightly  modifying  the proof of \cite[Theorem 2.5]{Toft-2012}
we conclude thay
\[
|V_\phi f(x,\xi)|\lesssim e^{\ep( M(|x|)+ M(|\xi|))},
\]
for some $\ep >0$, see also \cite[Proposition 3.2]{JPTT2}.
Since $V_\phi f(x,\xi)$ has compact support in the $x$-variable, it follows that
\[
|V_\phi f(x,\xi)|\lesssim e^{\ep M(|\xi|)}.
\]

\par

For the opposite direction, assume that $\supp V_\phi f\subseteq K\times \rr d$,
for some compact set $K$. %, and that \eqref{rastiopadanje2} holds for some $\ep >0$. Then
%%
%\begin{equation}\label{fhatest2}
%|\widehat{f}(\xi)|=\Big |\int V_\phi f(x,\xi)\, dx\Big|\lesssim e^{\ep|\xi|^{1/s}},\quad \text{for some}\  \ep >0.
%\end{equation}
%%
Assume that $\supp \phi \subseteq K$ and choose $\fy \in \cD^{(s)}(\rr d)$ such that
$\supp \fy \cap 2K =\emptyset$. Then
\[
(f, \fy)= (\nm \phi {L^2})^{-2}(V_\phi f, V_\phi \fy)=0,
\]
which implies that $f$ has compact support. Here the first equality is the
Moyal's identity (cf. \cite{Gro-book}). This implies that $f$ has compact support and the condition %Hence, \eqref{fhatest2} and the fact that
$f\in (\cD^{(M_p)})'(\rr d)$ now gives $f\in (\mathcal E^{(M_p)})'(\rr d)$.
\end{proof}

\subsection{Modulation spaces}
The modulation space norms traditionally  measure
the joint time-frequency distribution of $f\in \sch '$,
we refer, for instance, to \cite{Fe4}, \cite[Ch.~11-13]{Gro-book} and
the original literature quoted there for various properties and applications.
It is usually sufficient to observe modulation spaces with weights which admit at most polynomial growth at infinity.
However the study of ultra-distributions requires a more general approach that includes the weights of exponential or even superexponential growth, cf. \cite{CPRT1, Toft-2017}.
Note that the general approach introduced already in  \cite{Fe4}
includes the weights of sub-exponential growth.
We refer to \cite{FG1, FG2}
for related but even more general constructions, based on the general theory of coorbit spaces.

Depending on the
growth of the weight function $m$, different Gelfand-Shilov
classes may be chosen as fitting test function spaces for \modsp
s, see \cite{CPRT1,T2,Toft-2017}. The widest class of weights allowing to
define \modsp s is the weight class $\cN$. A weight function  $m$
on $\mathbb{R}^d$ belongs to $\cN$  if it is a continuous, positive
function such that
\begin{equation}\label{s12}
m(z)=o(e^{c z^2}),\,\quad\mbox{for}\,\,|z|\rightarrow\infty,\quad
\forall c>0,
\end{equation}
with $z\in\mathbb{R}^d$. For instance, every function $m(z)=e^{s |z|^b}$,
with $s>0$ and $0\leq b<2$,  is in $\cN$. Thus, the weight $m$ may
grow faster than exponentially at infinity. For example, the choice $ m \in \cN \setminus \cup_v \cM _v$,
when the weights $v$ satisfy the Beurling-Domar condition from Introduction, is related to the spaces of  quasianalytic functions, \cite{CPRT2}.
We notice that  there
is a limit in enlarging the weight class for \modsp s, imposed by
Hardy's theorem:
if $m(z)\geq C e^{c z^2}$, for some $c>\pi/2$, then the corresponding \modsp s  are trivial \cite{GZ01}.

\begin{definition}\label{defmodnorm}
Let $m\in \cN$, and $g$ a non-zero \emph{window} function in $\cS^{1/2}_{1/2}(\mathbb{R}^d)$. For
$1\leq p,q\leq \infty$ the {\it  modulation space} $M^{p,q}_m(\Ren)$ consists of all
$f\in (\cS^{1/2}_{1/2})' (\mathbb{R}^d)$
such that $V_gf\in L^{p,q}_m(\Renn )$
(weighted mixed-norm spaces). The norm on $M^{p,q}_m$ is
$$
\|f\|_{M^{p,q}_m}=\|V_{g}f\|_{L^{p,q}_m}=\left(\int_{\Ren}
  \left(\int_{\Ren}|V_{g} f(x,\o)|^pm(x,\o)^p\,
    dx\right)^{q/p}d\o\right)^{1/q}
$$
(with obvious changes if either $p=\infty$ or $q=\infty$). If
$p,q< \infty $, the \modsp\ $\Mmpq $ is the norm completion of
$\cS^{1/2}_{1/2}$ in the $\Mmpq $-norm. If $p=\infty $ or
$q=\infty$, then $\Mmpq $ is the completion of $\cS^{1/2}_{1/2}$
in the weak$^*$ topology.
\end{definition}

When $f,g \in \cS^{(1)} (\mathbb{R}^d )$, the above integral is convergent thanks to Theorem \ref{nec-suf-cond}.
Namely, for a given $ m\in\cM _v$ there exist $l>0$ such that
$ m (x,\o) \leq C e^{l \|\phas\|}$ and therefore
\begin{eqnarray*}
&& \left|\int_{\Ren}
  \left ( \int_{\Ren}|V_gf(x,\o)|^p m(x,\o)^p\,
    dx\right)^{q/p}d\o\right|\\
   && \quad\quad\quad\quad
     \leq
    C \left|\int_{\Ren}
  \left( \int_{\Ren}|V_gf(x,\o)|^p e^{l p\|\phas\|}\,
    dx\right) ^{q/p}  d\o\right| < \infty
\end{eqnarray*}
since by  Theorems \ref{nec-suf-cond} and Theorem \ref{GS-characterization}
we have $ |V_gf(x,\o)| < C e^{-s \|\phas\| } $ for every $s > 0.$
This implies $ \cS ^{(1)} \subset M^{p,q}_m.$

In particular, when $m$ is a polynomial weight of the form $m (x, \omega) = \langle  x \rangle ^t
\langle \omega \rangle ^s$ we will use the notation
$M^{p,q}_{s,t}(\mathbb{R}^d)$ for the modulation spaces which consists of all
$f\in \mathcal{S}'(\mathbb{R}^d)$ such that
%%
%\begin{equation} %\label{modnorm}
$$
\| f \|_{M^{p,q}_{s,t}} \equiv \left  ( \int _{\mathbb{R}^d} \left ( \int _{\mathbb{R}^d}
|V_\phi f(x,\omega )\langle  x \rangle ^t
\langle \omega \rangle ^s|^p\, dx  \right )^{q/p}d\omega  \right )^{1/q}<\infty
$$%\end{equation}
(with obvious interpretation of the integrals when $p=\infty$ or $q=\infty$).

If $p=q$, we write $M^p_m$ instead of $M^{p,p}_m$, and if $m(z)\equiv 1$ on $\Renn$, then we write $M^{p,q}$ and $M^p$ for $M^{p,q}_m$ and $M^{p,p}_m$, and so on.

In the next proposition we show that  $\Mmpq (\Ren )$ are  Banach spaces
whose definition is independent of the choice of the window
$g \in M^1_{v} \setminus \{ 0\}$.
In order to do so, we need the adjoint of the short-time Fourier transform.

%%%%%%%%%%%%%%%%%%%%%%%%%%%%%%
For given window $ g \in  \mathcal{S} ^{(1)} $ and a
function $ F (x,\xi) \in L^{p,q} _m (\R^{2d})$ we (formally) define $ V^* _g F $ by
$$
\langle V^* _g F, f \rangle := \langle F, V_g f \rangle.
$$

\begin{proposition} \label{emjedanve} Let $v$ be a submultiplicative weight.
Fix $m \in  \cM _v$ and $ g, \psi \in  \mathcal{S} ^{(1)},$ with $\langle g, \psi \rangle\not= 0$. Then
\begin{enumerate}
\item $ V^* _g :  L^{p,q} _m (\R^{2d}) \rightarrow  M^{p,q} _m (\R^{d}), $ and
\begin{equation}  \label{vstar}
\|  V^* _g F  \|_{ M^{p,q} _m } \leq C \| V_\psi g \|_{ L^{1} _v }   \| F \|_{L^{p,q} _m}.
\end{equation}
\item The inversion formula holds: $ I_{ M^{p,q} _m }  = \langle g, \psi \rangle^{-1}
  V^* _g   V _\psi,$ where  $ I_{ M^{p,q} _m } $ stands for the identity operator.
\item   $\Mmpq (\Ren )$ are  Banach spaces
whose definition is independent on the choice of
$ g \in \mathcal{S} ^{(1)} \setminus \{ 0 \} $.
\item The space of admissible windows can be extended from $ {\mathcal S} ^{(1)} $ to $M^1 _v .$
\end{enumerate}
\end{proposition}

\begin{proof} We refer to \cite{CPRT1} for the proof which is based on the proof of
\cite[Proposition 11.3.2.]{Gro-book}. Note that in  (4)  the density of $ \cS ^{(1)} $ in $ M^{p,q}_m$
is essential. This fact is not obvious, and we refer to  \cite{Elena07} for the proof.
Then we may proceed by using the standard arguments, cf. \cite[Theorem 11.3.7]{Gro-book}.
\end{proof}

The following theorem lists some basic properties of modulation spaces.
We refer to \cite{Fe4, Gro-book, GZ, PT1, T3, Toft-2012} for the proof.

\begin{theorem} \label{modproerties}
Let $p,q,p_j,q_j\in [1,\infty ]$ and $s,t,s_j,t_j\in \mathbb{R}$, $j=1,2$. Then:
\begin{enumerate}
\item $M^{p,q}_{s,t}(\mathbb{R}^d)$ are Banach spaces, independent of the choice of
$\phi \in \mathcal{S}(\mathbb{R}^d) \setminus 0$;

\item if  $p_1\le p_2$, $q_1\le q_2$, $s_2\le s_1$ and
$t_2\le t_1$, then
$$
\mathcal{S}(\mathbb{R}^d)\subseteq M^{p_1,q_1}_{s_1,t_1}(\mathbb{R}^d)
\subseteq M^{p_2,q_2}_{s_2,t_2}(\mathbb{R}^d)\subseteq
\mathcal{ S}'(\mathbb{R}^d);
$$

\item $ \displaystyle
\cap _{s,t} M^{p,q}_{s,t}(\mathbb{R}^d)=\mathcal{ S}(\mathbb{R}^d),
\quad
\cup _{s,t}M^{p,q}_{s,t}(\mathbb{R}^d)=\mathcal{ S}'(\mathbb{R}^d); $
\item Let $ 1\leq p, q \leq  \infty, $ and let $ w_s (x,\omega) = e^{s \|(x,\omega) \|},$
$x,\omega \in \mathbb{R}^d.$
Then
$$
\Sigma_1 ^1 (\mathbb{R}^d)=  {\mathcal S}^{(1)} (\mathbb{R}^d)=  \bigcap _{s \geq  0} M _{w_{s}} ^{p,q} (\mathbb{R}^d),\;\;\;
(\Sigma_1 ^1)' (\mathbb{R}^d)
%= {\mathcal S}^{(1)'} (\mathbb{R}^d)
= \bigcup _{s \geq 0} M _{1/w_{s}} ^{p,q} (\mathbb{R}^d),
$$
$$
{\mathcal S}_1 ^{1} (\mathbb{R}^d) = {\mathcal S}^{\{1\}} (\mathbb{R}^d) =  \bigcup _{s >  0} M _{w_{s}} ^{p,q} (\mathbb{R}^d),\;\;\;
({\mathcal S}_1 ^{1})' (\mathbb{R}^d)
%= {\mathcal S}^{\{1\}'} (\mathbb{R}^d)
= \bigcap _{s > 0} M _{1/w_{s}} ^{p,q} (\mathbb{R}^d).
$$
\item For   $p,q\in [1,\infty )$, the dual of $ M^{p,q}_{s,t}(\mathbb{R}^d)$ is
$ M^{p',q'}_{-s,-t}(\mathbb{R}^d),$ where $ \frac{1}{p} +  \frac{1}{p'} $ $ =
 \frac{1}{q} +  \frac{1}{q'} $ $ =1.$
\end{enumerate}
\end{theorem}

\begin{remark} \label{GSandmod}
In the context of quasianalytic Gelfand-Shilov spaces, we recall (a special case of) \cite[Theorem 3.9]{Toft-2012}:
Let $s,t> 1/2$ and set
$$
w_h (x, \omega) \equiv e^{h (|x|^{1/t} + |\omega|^{1/s})}, \;\;\; h>0, \; x,\omega \in \mathbb{R}^d.
$$
Then
$$
\Sigma_t ^s (\mathbb{R}^d)=  \bigcap _{h>  0} M _{w_{h}} ^{p,q} (\mathbb{R}^d),\;\;\;
(\Sigma_t ^s)' (\mathbb{R}^d) =  \bigcup _{h>0} M _{1/w_{h}} ^{p,q} (\mathbb{R}^d),
$$
$$
{\mathcal S}_t ^{s} (\mathbb{R}^d)
=  \bigcup _{h >  0} M _{w_{h}} ^{p,q} (\mathbb{R}^d),\;\;\;
({\mathcal S}_t ^{s})' (\mathbb{R}^d)
= \bigcap _{h > 0} M _{1/w_{h}} ^{p,q} (\mathbb{R}^d).
$$
\end{remark}

Modulation spaces include the following well-know function spaces:
\begin{enumerate}
\item $ M^2 (\mathbb{R}^d) = L^2 (\mathbb{R}^d),$  and $ M^2 _{t,0}(\mathbb{R}^d) = L^2 _t (\mathbb{R}^d);$
\item The Feichtinger algebra: $ M^1 (\mathbb{R}^d) = S_0 (\mathbb{R}^d);$
\item Sobolev spaces: $ M^2 _{0,s}(\mathbb{R}^d) = H^2 _s (\mathbb{R}^d) = \{ f \, | \,
\hat f (\omega) \langle \omega \rangle ^s \in  L^2 (\mathbb{R}^d)\};$
\item Shubin spaces: $ M^2 _{s}(\mathbb{R}^d) = L^2 _s (\mathbb{R}^d) \cap H^2 _s (\mathbb{R}^d) = Q_s (\mathbb{R}^d),$
cf. \cite{Shubin91}.
\end{enumerate}

%%%%%%%%

%%%%%%%%%%%%%%%%%%%%%%%%%%%%%%%%%%% 
\section{The invariance property of Wave-front sets} \label{sec4}
%%%%%%%%%%%%%%%%%%%%%%%%%%%%%%%%%%%

\par

Next we define wave-front sets with respect to modulation spaces and show that they
agree with corresponding wave-front sets of Fourier Lebesgue types. More precisely, we
prove that \cite[Theorem 6.1]{PTT} holds if the weights of polynomial growth
are replaced by more general submultiplicative weights.

\par

Let there be given a sequence $\{ M_p\}$ which satisfies $ (M.1) - (M.3)'$ and let  $ M(\rho) $ denote its associated function. Furthermore, let
$ p,q\in [1,\infty ]$, and  $\Gamma \subseteq \mathbb{R}^d\setminus 0$ be an open cone.
If $f\in (\cS^{(1)})'(\mathbb{R}^d)$ and $\omega \in \mathcal M_{M(\rho)} (\mathbb{R}^{2d})$,
then we define

\begin{multline}\label{modseminorm}
|f|_{\mathcal B(\Gamma )} = |f|_{\mathcal B(\phi ,\Gamma )}
\equiv
\Big ( \int _{\Gamma} \Big ( \int _{\mathbb{R}^{d}} | V_\phi f(x,\xi )\omega
(x,\xi )|^p\, dx\Big )^{q/p}\, d\xi \Big )^{1/q}
\\[1ex]
\text{when}\quad \mathcal B=M^{p,q}_{(\omega )}=M^{p,q}_{(\omega
)}(\mathbb{R}^d).
\end{multline}
We note that
$|f|_{\mathcal B(\Gamma )} = \nm
f{M^{p,q}_{(\omega )}}$ when $\Gamma
=\mathbb{R}^d\setminus 0$ and $\phi \in \mathcal{S} ^{(s)} (\mathbb{R}^d)$,
and that $ |f|_{\mathcal B(\phi ,\Gamma )}$ might attain $+\infty$.

\par

Furthermore, when $\mathcal B=M^{p,q}_{(\omega )}$, the sets $\Theta _{\mathcal B}(f)$, $\Sigma
_{\mathcal B}(f)$ and $\WF _{\mathcal B}(f)$ with respect to the modulation space
$\mathcal B$ are defined in the same way as in Section
\ref{sec2}, after replacing the semi-norms of Fourier Lebesgue types in
\eqref{skoff1}  with the semi-norms in \eqref{modseminorm}.

\par

\begin{proposition}\label{nezavisnost}
Let there be given a sequence of positive numbers
$ (M_p)_{p \in \mN_0} $ which satisfies $ (M.1) - (M.3)'$, and let $M(\rho)$, $\rho>0$, be its associated function.
If $f\in (\mathcal{D}^{(M_p)})' (\mathbb{R}^d) $
then
$\WF  _{M^{p,q}_{(\omega )}}(f)$ is independent of $p$ and
$\phi \in \mathcal{S}^{(M_p)}(\mathbb{R}^d)\setminus 0$ in \eqref{modseminorm} .
\end{proposition}

\begin{proof}
We may assume that $f\in (\mathcal E ^{ (M_p) })' (\mathbb{R}^d)$
and that $\omega (x,\xi )=\omega (\xi )$ since the statements only concern local assertions.

\par

We follow the idea of the proof of \cite[Theorem 3,1]{JPTT2}, and
in order to prove that $\WF _{M^{p,q}_{(\omega )}}(f)$ is independent of
$\phi \in \mathcal{S}^{(M_p)}(\mathbb{R}^d)\setminus 0$, we assume that $\phi ,\phi _1\in
 \mathcal{S}^{(M_p)}(\mathbb{R}^d)\setminus 0$ and let $|\cdo |_{\mathcal C_1(\Gamma )}$ be
the semi-norm in \eqref{modseminorm} after $\phi$ has been replaced by
$\phi _1$. Let $\Gamma _1$ and $\Gamma _2$ be open cones in $\mathbb{R}^d$
such that $\overline {\Gamma _2}\subseteq \Gamma _1$. The asserted
independency of $\phi$ follows if we prove that
\begin{equation}\label{est2.6}
|f|_{\mathcal C(\Gamma _2)} \le C(|f|_{\mathcal C_1(\Gamma _1)}+1),
\end{equation}
for some positive constant $C$.
Let
\[
\Omega _1=\sets {(x,\xi )}{\xi \in \Gamma _1}\subseteq \mathbb{R}^{2d}\quad
\text{and}\quad \Omega
_2=\complement \Omega _1\subseteq \mathbb{R}^{2d},
\]
with characteristic functions $\chi _1$ and $\chi _2$ respectively,
and set
$$
F_k(x,\xi )=|V_{\phi _1}f(x,\xi )|\omega (\xi )\chi _k(x,\xi ), \quad
k=1,2,
$$ and $  G=|V_\phi \phi _1(x,\xi )e^{M(|\xi| )}|$.
Since $\omega$ is $v$-moderate, it follows from \cite[Lemma 11.3.3]{Gro-book} that
\[
|V_\phi f(x,\xi )\omega (x,\xi )|\lesssim \big ( (F_1+F_2)*G\big )(x,\xi),
\]
which implies that
\begin{equation*}
|f|_{\mathcal C(\Gamma _2)} \lesssim J_1+J_2,
\end{equation*}
where
\[
J_k = \Big (\int _{\Gamma _2} \Big (\int_{\mathbb{R}^d} |(F_k*G)(x,\xi )|^p\, dx\Big
)^{q/p}\, d\xi \Big )^{1/q}, \quad k=1,2.
\]

\par

By Young's inequality
\begin{equation*}
J_1\le \nm {F_1*G}{L^{p,q}_1}\le \nm G{L^1}\nm {F_1}{L^{p,q}_1}
=C|f|_{\mathcal C_1(\Gamma _1)},
\end{equation*}
where $C=\nm G{L^1} = \nm {V_\phi \phi _1(x,\xi )e^{M(|\xi| )}}{L^1}<\infty$, in view of Proposition \ref{stftestimates}.

\par

Next we consider $J_2$.
For $\xi\in \Gamma_2$ fixed and integrating over $\eta\in \complement \Gamma_1$,
it follows from  Propositon \ref{stftestimates} and Lemma \ref{osobineasocirane}
that for some $\ep >0$ and every $N,h>0$ we have that
$ |(F_2 * G)(x,\xi )| $ is bounded by
$$  C \iint _{\mathbb{R}^{2d}}
e^{-N M(|y|)} e^{\ep M( | \eta |)}  e^{- h (M(|x-y|) + M(|\xi - \eta|))} e^{M(|\xi-\eta|)} \, dy d\eta,
$$
for some constant $C>0.$
Therefore, there exist a constant $c>0$ such that
\begin{multline*}
|(F_2 * G)(x,\xi )|
\\[1ex]
\lesssim \iint _{\mathbb{R}^{2d}}
e^{-N M(|y|)} e^{\ep M(| \eta |)}
e^{- hM(|x-y|)- hc(M(|\xi|) + M(| \eta|))}
e^{(M(|\xi|) + M(|\eta|) )} \, dy d\eta
\\[1 ex]
\lesssim   e^{ (-N+h) M(|x|)}  e^{(1 -hc)M(|\xi|)}
\iint _{\mathbb{R}^{2d}}  e^{-h M(|y|)}
e^{(1+\ep  -hc)M(|\eta|) } \, dy d\eta,
\\[1 ex]
\lesssim    e^{ (-N+h) M(|x|)}  e^{(1 -hc)M(|\xi|)} < \infty ,
\end{multline*}
since $N>0$ and $h$ can be chosen arbitrarily.
Therefore
\begin{multline*}
J_2  = \Big (\int _{\Gamma _2} \Big (\int_{\mathbb{R}^d} |(F_2 * G)(x,\xi )|^p\, dx\Big )^{q/p}\, d\xi \Big )^{1/q}
\\[1 ex]
\lesssim
 \Big (\int _{\Gamma _2} \Big (\int_{\mathbb{R}^d}\Big (
 e^{ (-N+h) M(|x|)}  e^{(1 -hc) M(|\xi|)}
 \Big)^p\, dx\Big )^{q/p}\, d\xi \Big )^{1/q}<\infty.
\end{multline*}
This proves that \eqref{est2.6}, and hence $\WF _{\mathcal C}(f)$ is
independent of $\phi \in \mathcal{S}^{(s)}(\mathbb{R}^d)\setminus 0$.
\end{proof}

\par

The main result of this section,  Theorem \ref{wavefrontsequal}, now follows from Proposition \ref{nezavisnost}
and calculations given in the proof of \cite[Theorem 3.1]{JPTT2}. For that reason we omit the proof.

\begin{theorem}\label{wavefrontsequal}
Let there be given a sequence of positive numbers
$ (M_p)_{p \in \mN_0} $ which satisfies $ (M.1) - (M.3)'$, and let $M(\rho)$, $\rho>0$ be its associated function.
Let  $p,q\in [1,\infty ]$ and  $\omega \in \mathcal{M} _{M(\rho)} (\mathbb{R}^{2d})$.
If $f\in (\mathcal{D}^{(M_p)})' (\mathbb{R}^d) $
then
\begin{equation}\label{WFidentities2}
\WF _{\mathcal FL^q_{(\omega)}}(f)= \WF _{M^{p,q}_{(\omega )}}(f).
\end{equation}
\end{theorem}

\par

Finally, note that for a  given  sequence of positive numbers
$ (M_p)_{p \in \mN_0} $ which satisfies $ (M.1) - (M.3)'$, and its associated function $M(\rho)$, $\rho>0$,
when $p,q\in [1,\infty ]$,  $\omega \in \mathcal{M} _{M(\rho)} (\mathbb{R}^{2d})$ and
$f\in  (\mathcal{E}^{ (M_p)})' (\mathbb{R}^d)$, then it follows from the definition of wave-front sets that
then
\[
f \in \mathcal B \quad
\Longleftrightarrow \quad
\WF _{\mathcal B }(f) =\emptyset,
\]
when $\mathcal B$ is equal to $\mathcal FL^{q}_{(\omega)}$ or
$M^{p,q}_{(\omega )}$. In particular, by Theorem \ref{wavefrontsequal} we obtain
\[\mathcal FL^q _{(\omega)} \cap (\mathcal E^{(M_p)})'(\mathbb{R}^d)=\Mopq \cap (\mathcal E^{(M_p)})'(\mathbb{R}^d),\]
and we recover Corollary 6.2 in \cite{PTT}, Theorem 2.1 and Remark 4.6 in \cite{RSTT}.

\section*{Acknowledgement}

This work is supported by MPNTR through Project 174024 and DS028 -- Tifmofus.

\bibliographystyle{amsplain}

\end{document}